\PassOptionsToPackage{dvipsnames,table}{xcolor}
\documentclass[11pt]{article}

\usepackage[utf8]{inputenc}
\usepackage[T1]{fontenc}
\usepackage{lmodern}
\usepackage[margin=1in]{geometry}
\usepackage[round,authoryear]{natbib}
\setcitestyle{authoryear,round,citesep={;},aysep={,},yysep={;}}
\usepackage{amsmath,amssymb,amsthm,mathtools}
\usepackage{bm}
\usepackage{booktabs,multirow,array,tabularx}
\usepackage{graphicx}
\usepackage{xcolor}
\usepackage{tikz}
\usetikzlibrary{arrows.meta,positioning,calc,fit,backgrounds,shapes.geometric}
\usepackage{algorithm}
\usepackage[noend]{algpseudocode}
\usepackage{enumitem}
\usepackage{microtype}
\usepackage{hyperref}
\usepackage{url}
\usepackage[capitalize,nameinlink]{cleveref}
\usepackage{thmtools,thm-restate}
\usepackage{caption}
\usepackage{subcaption}
\usepackage{etoolbox}
\usepackage{placeins}

\hypersetup{colorlinks=true,linkcolor=RoyalBlue,citecolor=OliveGreen,urlcolor=Maroon,
  pdftitle={Saddle-Point Problems with a Low-Dimensional Block Do Not Need Accurate Inner Solves},
  pdfauthor={Ivan Fomin, Alexander V. Gasnikov}}
\theoremstyle{plain}
\newtheorem{theorem}{Theorem}
\newtheorem{proposition}[theorem]{Proposition}
\newtheorem{lemma}[theorem]{Lemma}
\newtheorem{corollary}[theorem]{Corollary}

\theoremstyle{definition}

\newtheorem{assumption}{Assumption}

\theoremstyle{remark}
\newtheorem{remark}{Remark}

\crefname{theorem}{Theorem}{Theorems}\Crefname{theorem}{Theorem}{Theorems}
\crefname{lemma}{Lemma}{Lemmas}\Crefname{lemma}{Lemma}{Lemmas}
\crefname{proposition}{Proposition}{Propositions}\Crefname{proposition}{Proposition}{Propositions}
\crefname{corollary}{Corollary}{Corollaries}\Crefname{corollary}{Corollary}{Corollaries}
\crefname{conjecture}{Conjecture}{Conjectures}\Crefname{conjecture}{Conjecture}{Conjectures}
\crefname{assumption}{Assumption}{Assumptions}\Crefname{assumption}{Assumption}{Assumptions}
\crefname{definition}{Definition}{Definitions}\Crefname{definition}{Definition}{Definitions}
\crefname{remark}{Remark}{Remarks}\Crefname{remark}{Remark}{Remarks}
\crefname{algorithm}{Algorithm}{Algorithms}\Crefname{algorithm}{Algorithm}{Algorithms}
\crefname{section}{Section}{Sections}\Crefname{section}{Section}{Sections}
\crefname{subsection}{Section}{Sections}\Crefname{subsection}{Section}{Sections}
\crefname{appendix}{Appendix}{Appendices}\Crefname{appendix}{Appendix}{Appendices}
\crefname{subappendix}{Appendix}{Appendices}\Crefname{subappendix}{Appendix}{Appendices}
\crefname{figure}{Figure}{Figures}\Crefname{figure}{Figure}{Figures}
\crefname{table}{Table}{Tables}\Crefname{table}{Table}{Tables}
\crefname{enumi}{item}{items}\Crefname{enumi}{Item}{Items}
\crefformat{equation}{#2(#1)#3}\Crefformat{equation}{#2(#1)#3}

\algrenewcommand\algorithmiccomment[1]{\hfill\(\triangleright\)~\textit{#1}}
\AtBeginEnvironment{algorithmic}{\small}

\newcommand{\RR}{\mathbb{R}}
\newcommand{\eps}{\varepsilon}
\newcommand{\cX}{\mathcal{X}}
\newcommand{\cY}{\mathcal{Y}}

\newcommand{\nrm}[1]{\|#1\|}
\newcommand{\ip}[2]{\langle #1,#2\rangle}
\newcommand{\Gap}{\operatorname{Gap}}
\newcommand{\proj}{\operatorname{proj}}
\newcommand{\aff}{\operatorname{aff}}
\newcommand{\vol}{\operatorname{vol}}
\newcommand{\diam}{\operatorname{diam}}
\newcommand{\cg}{\operatorname{cg}}
\newcommand{\argmin}{\operatorname*{arg\,min}}
\newcommand{\argmax}{\operatorname*{arg\,max}}
\newcommand{\logp}{\log_{+}}
\newcommand{\Ox}{\mathsf{O}_x}
\newcommand{\Oy}{\mathsf{O}_y}
\newcommand{\Tx}{T_x}
\newcommand{\Ty}{T_y}
\newcommand{\Otil}{\widetilde{O}}

\newcommand{\My}{M_y}

\newcommand{\Ry}{R_y}
\newcommand{\Rx}{R_x}

\newcommand{\defeq}{\coloneqq}
\newcommand{\T}{^{\top}}

\definecolor{boxblue}{RGB}{232,240,252}
\definecolor{boxgreen}{RGB}{234,246,236}
\definecolor{boxorange}{RGB}{253,241,226}

\newcommand{\takeaway}[1]{\par\smallskip\noindent\colorbox{boxgreen}{\parbox{\dimexpr\linewidth-2\fboxsep}{\small\textbf{Takeaway.} #1}}\par\smallskip}

\makeatletter
\providecommand{\theHALG@line}{}
\renewcommand{\theHALG@line}{\thealgorithm.\arabic{ALG@line}}
\makeatother

\title{Saddle-Point Problems with a Low-Dimensional Block\\Do Not Need Accurate Inner Solves}

\author{Ivan Fomin\\Yandex, Moscow, Russia
\and
Alexander V. Gasnikov\\Innopolis University, Innopolis, Russia}

\date{}

\begin{document}

\maketitle

\begin{abstract}
Many learning problems couple a high-dimensional block of parameters with a handful of adversarial or dual variables: worst-group risk over a few groups, or learning under a few constraints. We consider $\min_{x\in\mathcal X}\max_{y\in\mathcal Y} f(x,y)$ with $\mathcal X\subseteq\mathbb{R}^n$, a compact convex set $\mathcal Y\subset\mathbb{R}^m$, and $m\ll n$, where every slice $f(\cdot,y)$ is smooth and strongly convex with condition number $\kappa$, and we count the oracle calls for $x$ and for $y$ separately. The textbook approach runs a cutting-plane method in $y$ and solves every inner problem to accuracy $\varepsilon$ with an accelerated method; it needs $O(m\log(1/\varepsilon))$ calls for $y$ but $O(m\sqrt{\kappa}\log^2(1/\varepsilon))$ calls for $x$. We show that the inner problems need not be solved accurately. Our method, certificate transport, keeps a strongly convex lower model of a single slice and uses it as a prior for a short accelerated run on the next slice. By concavity in $y$, every call for $y$ then either cuts the localizer or moves the lower model to a mixture of the two slices with a certified increase of the lower bound. For $m=1$ this gives an $\varepsilon$-saddle point after $O(\sqrt{\kappa}\log(1/\varepsilon))$ calls for $x$, up to an initialization term, and $O(\log(1/\varepsilon))$ calls for $y$; both counts are optimal, even though $f(x,\cdot)$ is only assumed to be concave and Lipschitz. For general $m$, with centers of gravity of the localizer treated as computable, the method needs $O((m+\sqrt{m\kappa})\log(1/\varepsilon))$ calls for $x$ and the optimal $O(m\log(1/\varepsilon))$ calls for $y$. Under strong concavity in $y$, a two-point accelerated method makes the number of calls for $x$ independent of $m$.
\end{abstract}


\section{Introduction}\label{sec:intro}

We consider the saddle-point problem
\begin{equation}\label{eq:problems}
\min_{x\in\cX}\ \max_{y\in\cY} f(x,y), \qquad x\in\RR^n,\ \ y\in\RR^m,\ \ m\ll n,
\end{equation}
in which the two blocks of variables differ greatly in size: $n$ is large, while $m$ is small. For every $y$, the function $f(\cdot,y)$ is $\mu$-strongly convex and $L$-smooth with condition number $\kappa=L/\mu$. For every $x$, the function $f(x,\cdot)$ is concave, with supergradients bounded in norm by $M_y$ on a convex compact set $\cY$ of diameter $R_y$. No smoothness in $y$ is assumed.

Problems of this form arise, for example, in learning under a few constraints \citep{agarwal2018reductions,cotter2019two}. There $f$ is the Lagrangian, $y$ is the vector of Lagrange multipliers, bounded as usual under Slater's condition, and $m$ is the number of constraints. Another example is group distributionally robust learning \citep{sagawa2020distributionally}, where $y$ weights a few groups. In both examples $f$ is linear in $y$, so the dual function is differentiable under strong convexity in $x$, but its smoothness constant, of order $G^2/\mu$ with $G$ a bound on the Jacobian of the losses, is huge for small ridge parameters, and nonsmooth regularizers of $y$ (e.g., $\ell_1$ or total-variation penalties on the group weights) make it nonsmooth. Cutting planes are insensitive to both, so we assume no smoothness in $y$.

Each block calls for its own method. For the large block, accelerated gradient descent is optimal and dimension-free \citep{nesterov1983method}. For the small block, the center-of-gravity method needs only $O(m\log\frac1\eps)$ queries and requires no smoothness at all \citep{levin1965algorithm,newman1965location,nemirovski1983problem}. The two blocks are also accessed through different computations: a gradient in $x$ costs a backward pass through the model, whereas a query in $y$ returns only a function value and a supergradient in a low-dimensional space. We therefore count the calls to the two oracles separately, as $\Tx$ and $\Ty$. The question is how to combine the two methods so that both counts are small.

\paragraph{The obvious approach and its cost.}
The textbook solution nests one method inside the other. Cutting planes are run on the dual function $\phi(y)=\min_xf(x,y)$, and every cut is computed by running accelerated gradient descent in $x$ to high accuracy. With $B=M_yR_y$, this gives $\Ty=O(m\log\frac B\eps)$ but $\Tx=O(m\sqrt\kappa\log\frac B\eps\log\frac{\mu R_x^2+B}\eps)$, where $R_x$ is the diameter of $\cX$ (\citealp{gladin2023solving}; see \eqref{eq:nested}). The waste is plain: each inner problem is solved from scratch, and its solution is discarded as soon as $y$ moves.

\paragraph{Our idea.}
The inner problems at different $y$ are closely related, so the work spent on one of them can be reused. Suppose we hold a quadratic lower model $q_b\le f(\cdot,y_b)$, which certifies $\phi(y_b)\ge b$. We run a block of $K$ accelerated gradient steps on $f(\cdot,y_c)$ at the current center of the localizer $y_c$, using $q_b$ as a prior, and make a single $y$-query. By concavity in $y$, this yields a lower model of the same form at $y^+=(1-\vartheta)y_b+\vartheta y_c$, with a bound $b^+$ that increases whenever the queried value exceeds the current threshold. We call this \emph{certificate transport}. With it, no inner problem is ever solved to accuracy $\eps$: every $y$-query either shrinks the localizer or raises the certified lower bound.

\subsection{Contributions}

\begin{enumerate}[leftmargin=1.3em,itemsep=2pt,topsep=2pt]
\item \textbf{Certificate transport.} We introduce a mechanism that couples cutting planes in $y$ with blocks of $K$ accelerated gradient steps in $x$, without ever solving an inner problem to high accuracy. It requires only concavity and Lipschitz continuity in $y$. The method maintains a computable bound on the saddle-point gap and returns a pair $(\bar x,y_b)$ certified by it (\cref{thm:general}). Its correctness does not depend on $K$: the block length affects only the running time.

\item \textbf{Complexity for any block length} (\cref{thm:general}). For every $K\ge1$ the method reaches gap $\eps$ with
\[
\Ty=O\Big(\big(m+\tfrac{\kappa}{K^2}\big)\log\tfrac{B}{\eps}\Big),
\qquad
\Tx=O\Big(\sqrt\kappa\,\big[1+\log_+\tfrac{\mu R_x^2}{B}\big]\Big)+O\Big(\big(mK+\tfrac{\kappa}{K}\big)\log\tfrac{B}{\eps}\Big),
\]
where the first term in $\Tx$ is a one-time initialization cost. The choice $K=\sqrt{\kappa/m}$ gives $\Ty=O(m\log\frac B\eps)$, the optimal count for cutting planes, and $\Tx=O(\sqrt{m\kappa}\log\frac B\eps)$ plus initialization. The product $m\sqrt\kappa$ of the nested scheme is thus replaced by the geometric mean $\sqrt{m\kappa}$, and the squared logarithm disappears. For $m=1$, both counts match the known lower bounds for each block separately. The same choice of $K$ minimizes our upper bound on any weighted cost $c_x\Tx+c_y\Ty$ up to a factor of three (\cref{pr:3opt}).

\item \textbf{Independence of $m$ under strong concavity} (\cref{sec:strong}). If $f$ is also strongly concave in $y$ and the coupling between the blocks is bounded by $\chi$, a two-point accelerated method uses $\Tx=O(\sqrt{\kappa+\chi}\,\Lambda)$ gradients, independently of $m$, and $\Ty=O(m\sqrt{1+\chi}\,\Lambda^2)$, with $\Lambda=1+\log\frac{B(1+\chi)}\eps$. An accelerated method on the primal function has the same $\Tx$ but solves $O(\sqrt{\kappa+\chi})$ instead of $O(\sqrt{1+\chi})$ problems in $y$ per halving of the error.
\end{enumerate}

Like the classical center-of-gravity method, our analysis treats the computation of centers of gravity in $\RR^m$ as free and counts only oracle calls. Replacing them with implementable centers, as in Vaidya's method \citep{vaidya1996new} or random-walk centroids \citep{bertsimas2004solving}, is a natural next step.

\paragraph{Related work.}
Nested schemes for problems with a low-dimensional block were analyzed by \citet{gladin2021solving,gladin2023solving,gladin2022vaidya}, and inexact oracles and accuracy certificates in such schemes by \citet{gladin2021mixed,gladin2023accuracy}. Our localizer $\{U\ge h\}$ is the level set of the level method \citep{lemarechal1995new}. Bundle methods with inexact oracles \citep{kiwiel2006proximal,deoliveira2014convex,deoliveira2014level,vanackooij2014level} and inexact dual decomposition \citep{necoara2014rate} are the classical answer to Lagrangians with inexact inner solutions; they control the error of every inner solve, whereas certificate transport never solves an inner problem to a prescribed accuracy and reuses its partial solution instead. Warm starts across outer iterations are standard in bilevel optimization \citep{ji2021bilevel}, but their analyses need the inner solution map to be Lipschitz, which we do not assume. Nested minimax methods that remove logarithmic factors \citep{lin2020near,kovalev2022optimal} require smoothness in both blocks. \Cref{alg:block} is the method of similar triangles \citep{gasnikov2018universal} with an arbitrary quadratic prior.

\paragraph{Organization.}
\Cref{sec:setting} states the assumptions, \cref{sec:ct,sec:strong} present the two methods, and \cref{sec:experiments} reports experiments. Proofs are in \cref{app:toolbox,app:minmax-general,app:minmax-cross}, experimental details in \cref{app:experiments}, and a detailed discussion of related work in \cref{app:related}.

\section{Setting}\label{sec:setting}

\paragraph{Problem and accuracy.}
Let $\cX\subseteq\RR^n$ be a closed convex set and $\cY\subset\RR^m$ a convex body, that is, a compact convex set with nonempty interior in its affine hull; we assume $m=\dim\aff\cY$. We consider the problem
\[
\min_{x\in\cX}\ \max_{y\in\cY} f(x,y),
\]
where $f(\cdot,y)$ is convex for every $y$ and $f(x,\cdot)$ is concave for every $x$. Let
\[
P(x)\defeq\max_{y\in\cY}f(x,y),\qquad \phi(y)\defeq\min_{x\in\cX}f(x,y).
\]
A candidate pair $(\hat x,\hat y)\in\cX\times\cY$ is measured by the saddle-point gap
\begin{equation}\label{eq:gap}
\Gap(\hat x,\hat y)\defeq P(\hat x)-\phi(\hat y)\ \ge 0,
\end{equation}
which bounds the suboptimality of both $\hat x$ in the primal problem and $\hat y$ in the dual problem.

\paragraph{Split oracles.}
An algorithm accesses $f$ only through two oracles, queried at points of $\cX\times\cY$:
\begin{equation}\label{eq:oracles}
\Ox(x,y)=\nabla_xf(x,y),
\qquad
\Oy(x,y)=\big(f(x,y),\,g\big),\quad g\in\partial_yf(x,y),
\end{equation}
where $\partial_y f$ denotes the superdifferential in $y$. We denote by $\Tx$ and $\Ty$ the numbers of calls to $\Ox$ and $\Oy$. As in information-based complexity \citep{nemirovski1983problem}, all other computations are free: projections onto $\cX$, operations with the finitely many planes and quadratics collected so far, and centers of gravity of convex bodies in $\RR^m$.

\paragraph{Assumptions.}
Our main results hold under the following assumption.

\begin{assumption}\label{ass:minmax}
The set $\cX$ has diameter at most $\Rx$, and $\cY$ has diameter at most $\Ry$. For every $y\in\cY$, the function $f(\cdot,y)$ is $\mu$-strongly convex and $L$-smooth on $\cX$ with $\mu>0$. For every $x\in\cX$, the function $f(x,\cdot)$ is concave on $\cY$, and the supergradients returned by $\Oy$ have norm at most $\My$.
\end{assumption}

We write $\kappa\defeq L/\mu$ and $B\defeq\My\Ry$. The constant $B$ bounds the variation of $f(x,\cdot)$ over $\cY$ and sets the scale of the gap. \Cref{ass:minmax} requires neither smoothness nor strong concavity in $y$, and nothing links the two blocks. The bound $\Rx$ is used only in the initialization (\cref{lem:init}); for $\cX=\RR^n$ it suffices that $\Rx$ bounds the initial distance to $\argmin f(\cdot,y_0)$.

In \cref{sec:strong} we use the following additional assumption.

\begin{assumption}[strong concavity and cross-Lipschitz coupling]\label{ass:cross}
For every $x\in\cX$, the function $f(x,\cdot)$ is $\nu$-strongly concave on $\cY$ with $\nu>0$, and for all $x\in\cX$ and $y,y'\in\cY$,
\[
\nrm{\nabla_xf(x,y)-\nabla_xf(x,y')}\le\ell\,\nrm{y-y'}.
\]
\end{assumption}

We set $\chi\defeq\ell^2/(\mu\nu)$. This quantity is scale-invariant and plays the role of a condition number for the coupling between the blocks.

\paragraph{Baseline.}
Running a cutting-plane method on the concave function $\phi$ and computing each cut by an accelerated method that solves the problem in $x$ to accuracy $\eps$ \citep{gladin2023solving} gives
\begin{equation}\label{eq:nested}
\Tx^{\rm nest}=O\Big(m\sqrt\kappa\,\log\frac B\eps\,\log\frac{\mu\Rx^2+B}{\eps}\Big),\qquad \Ty^{\rm nest}=O\Big(m\log\frac B\eps\Big).
\end{equation}
In the rest of the paper we show how to replace the factor $m\sqrt\kappa$ by $\sqrt{m\kappa}$ and remove the second logarithm.
\section{Certificate transport}\label{sec:ct}

In this section we assume only \cref{ass:minmax}. Nothing is known about $f(x,\cdot)$ beyond concavity and a Lipschitz bound: it may be piecewise linear, and its maximizers may lie on the boundary of $\cY$.

\paragraph{What the algorithm stores.}
The algorithm keeps two kinds of information. The first is a collection of \emph{upper planes} $\ell_i(y)=v_i+\ip{g_i}{y-y_i}$, each obtained from one call $\Oy(u_i,y_i)$. Since $f(u_i,\cdot)$ is concave, $\ell_i\ge f(u_i,\cdot)\ge\phi$ on $\cY$, and therefore
\[
U\defeq\min_i\ell_i\ \ge\ \phi,
\qquad
a\defeq\max_\cY U\ \ge\ \max_\cY\phi .
\]
The second is a single \emph{lower model}
\begin{equation}\label{eq:lowermodel}
f(x,y_b)\ \ge\ q_b(x)\defeq b+\tfrac\sigma2\nrm{x-x_b}^2\quad\text{for all }x\in\cX,\qquad \sigma\defeq\mu/2,
\end{equation}
which certifies $\phi(y_b)\ge b$. By the minimax theorem, some convex combination $\bar x$ of the points $u_i$ satisfies $P(\bar x)\le a$ (\cref{lem:recovery}), hence $\Gap(\bar x,y_b)\le a-b$. The algorithm reduces this \emph{certified width} $a-b$ to $\eps$.

\paragraph{The transport step.}
The key ingredient is the following observation (\cref{lem:block}, proved in \cref{app:toolbox}). Run $K$ steps of an accelerated method on $F=f(\cdot,y_c)$, using the lower model $q_b$ of a \emph{different} slice as a prior (\cref{alg:block}). The method returns two points $u_K,z_K\in\cX$ and, with $A=K(K+3)/(8\kappa)$ and $v=F(u_K)$, guarantees
\begin{equation}\label{eq:transport}
f\big(x,\,\underbrace{(1-\vartheta)y_b+\vartheta y_c}_{y^+}\big)\ \ge\ \underbrace{\frac{b+Av}{1+A}}_{b^+}+\frac\sigma2\nrm{x-z_K}^2\quad\forall x\in\cX,\qquad \vartheta=\frac{A}{1+A}.
\end{equation}
The proof combines the estimate-sequence invariant $b+A\,F(u_K)\le\min_\cX\Psi_K$ with concavity in $y$, namely $f(x,y^+)\ge(1-\vartheta)f(x,y_b)+\vartheta F(x)$. Only one function value, $v=F(u_K)$, is needed, and it comes together with a new upper plane from the same call to $\Oy$.

\begin{algorithm}[t]
\caption{A block of accelerated steps with a prior model: $\textsc{Block}(F,\pi,K)$}\label{alg:block}
\begin{algorithmic}[1]
\Require function $F$, $\mu_F$-strongly convex and $L_F$-smooth on $\cX$; prior model $\pi(x)=c+\frac\gamma2\nrm{x-\bar x}^2$; block length $K\ge1$
\State $A_0\gets0$,\ \ $\Gamma_0\gets\gamma$,\ \ $u_0\gets\bar x$,\ \ $z_0\gets\bar x$,\ \ $r_0\gets\gamma\bar x$
\For{$i=1,\dots,K$}
  \State $a_i\gets\frac{(i+1)\gamma}{2L_F}$,\ \ $A_i\gets A_{i-1}+a_i$,\ \ $\Gamma_i\gets\Gamma_{i-1}+\mu_F a_i$
  \State $x_i\gets\frac{A_{i-1}u_{i-1}+a_iz_{i-1}}{A_i}$
  \State $g_i\gets\nabla F(x_i)$ \Comment{one call to $\Ox$ if $F=f(\cdot,y)$}
  \State $r_i\gets r_{i-1}+a_i(\mu_F x_i-g_i)$,\ \ $z_i\gets\Pi_\cX\big(r_i/\Gamma_i\big)$
  \State $u_i\gets\frac{A_{i-1}u_{i-1}+a_iz_i}{A_i}$
\EndFor
\State \Return $(u_K,\,z_K,\,A_K)$, where $A_K=\frac{\gamma K(K+3)}{4L_F}$
\end{algorithmic}
\end{algorithm}

\begin{algorithm}[t]
\caption{Certificate transport (CT)}\label{alg:ct}
\begin{algorithmic}[1]
\Require $\eps>0$; block length $K\ge1$ (default $K=\lceil\sqrt{8\kappa/m}\,\rceil$)
\State Build a lower model $(b,x_b,y_b)$ and one plane with $a-b\le2B$ (\cref{cor:start})
\State $W\gets a-b$
\State $h\gets b+W/2$
\While{$a-b>\eps$}
  \If{$a-b\le 2W/3$}
    \State $W\gets a-b$, $h\gets b+W/2$ \Comment{new epoch}
  \EndIf
  \State $y_c\gets$ center of gravity of $\{y\in\cY:\ U(y)\ge h\}$
  \State $(u,z,A)\gets\textsc{Block}(f(\cdot,y_c),q_b,K)$ \Comment{$K$ calls to $\Ox$}
  \State $(v,g)\gets\Oy(u,y_c)$;\quad add the plane $\ell(y)=v+\ip{g}{y-y_c}$ \Comment{one call to $\Oy$}
  \If{$v>h$} \Comment{certificate transport}
     \State $b\gets\frac{b+Av}{1+A}$,\ \ $x_b\gets z$,\ \ $y_b\gets\frac{y_b+Ay_c}{1+A}$
  \EndIf
  \State $a\gets\max_\cY U$
\EndWhile
\State \Return $(\bar x,y_b)$, where $\bar x=\sum_i\lambda_iu_i$ and $\lambda\in\argmin_{\lambda\in\Delta}\max_{y\in\cY}\sum_i\lambda_i\ell_i(y)$
\end{algorithmic}
\end{algorithm}

The run is divided into epochs. In an epoch with initial width $W$, the threshold is $h=b+W/2$, and the algorithm queries the center of gravity $y_c$ of the localizer $\{y\in\cY:U(y)\ge h\}$. Every query makes progress in one of two ways. If $v\le h$, the new plane satisfies $\ell(y_c)=v\le h$ and cuts off a half-space whose boundary passes through the center of gravity of the localizer; by Grünbaum's lemma, the volume of the localizer shrinks by a factor of at least $(1-e^{-1})^{-1}$ (\cref{lem:grunbaum}). If $v>h$, the certificate is transported, and $b$ grows by $\vartheta(v-b)>\vartheta W/6$. Within an epoch $b$ can grow by at most $W/3$, so each epoch contains $O(1/\vartheta)=O(1+\kappa/K^2)$ such queries. In both cases the new plane can only decrease $U$, so the volume of the localizer does not grow within an epoch.

The width shrinks by a factor of at least $\frac32$ per epoch, so there are $O(\log\frac B\eps)$ epochs in total. At a change of epoch, the volume of the localizer grows by at most a factor $6^m$, and since $U$ is Lipschitz, the volume never drops below $(\eps/6B)^m\operatorname{vol}\cY$. Comparing these bounds shows that the total number of queries with $v\le h$ is $O(m\log\frac B\eps)$ (\cref{app:minmax-general}). Each query costs $K$ gradients, so the total number of gradients is
\begin{equation}\label{eq:balance}
\Big(\underbrace{m\cdot K}_{\text{cuts}}\;+\;\underbrace{\Big(1+\frac{\kappa}{K^2}\Big)\cdot K}_{\text{transports}}\Big)\cdot O\Big(\log\frac B\eps\Big),
\end{equation}
and the choice $K=\sqrt{\kappa/m}$ balances the two terms at $O(\sqrt{m\kappa}\log\frac B\eps)$.

\begin{theorem}\label{thm:general}
Let \cref{ass:minmax} hold and $0<\eps<2B$. For every $K\ge1$, \cref{alg:ct} returns $(\bar x,\hat y)\in\cX\times\cY$ with $\Gap(\bar x,\hat y)\le\eps$ after
\[
\Ty=O\Big(\Big(m+\frac{\kappa}{K^2}\Big)\log\frac B\eps\Big),
\qquad
\Tx=I_0+O\Big(\Big(mK+\frac{\kappa}{K}\Big)\log\frac B\eps\Big)
\]
oracle calls, where $I_0=O\big(\sqrt\kappa\,[1+\logp\frac{\mu\Rx^2}{B}]\big)$ is the cost of initialization. In particular, for $K=\lceil\sqrt{8\kappa/m}\,\rceil$,
\[
\Tx+\Ty=I_0+O\Big(\big(m+\sqrt{m\kappa}\big)\log\frac{B}{\eps}\Big),
\qquad
\Ty=O\Big(m\log\frac B\eps\Big).
\]
\end{theorem}

The proof is given in \cref{app:minmax-general}. The correctness of the output does not depend on $K$: the block length affects only the number of iterations.

Compared with the nested scheme \eqref{eq:nested}, \cref{thm:general} improves $\Tx$ by a factor of order $\min\{\sqrt m,\sqrt\kappa\}\cdot\log\frac1\eps$ and keeps the optimal $\Ty$.

\paragraph{The one-dimensional case.}
For $m=1$, the set $\cY$ is a segment and its center of gravity is the midpoint. \Cref{thm:general} then gives the following.

\begin{corollary}\label{cor:onedim}
Let \cref{ass:minmax} hold with $m=1$, and let $0<\eps<2B$. Then \cref{alg:ct} with $K=\lceil\sqrt{8\kappa}\,\rceil$ returns $(\bar x,\hat y)$ with $\Gap(\bar x,\hat y)\le\eps$ after
\[
\Tx=O\Big(\sqrt\kappa\,\Big[1+\logp\frac{\mu\Rx^2}{B}+\log\frac B\eps\Big]\Big),
\qquad
\Ty=O\Big(\log\frac B\eps\Big)
\]
oracle calls.
\end{corollary}

Both bounds match, in order, the known lower bounds for each block separately \citep{nemirovski1983problem,nesterov2018lectures} (for $\Tx$, when $n$ is large enough). Thus a merely concave, nonsmooth one-dimensional $y$-block costs nothing beyond its own complexity.

\paragraph{Choosing $K$ for a weighted cost.}
Suppose a call to $\Ox$ costs $c_x$, a call to $\Oy$ costs $c_y$, and let $r=c_y/c_x$. Up to the initialization term, \cref{thm:general} bounds the total cost $c_x\Tx+c_y\Ty$ by $c_x\,g(K)\log\frac B\eps$, where
\[
g(K)=\Big(\alpha+\frac{\beta}{K^2}\Big)(K+r),\qquad \alpha\asymp m,\quad \beta\asymp\kappa .
\]

\begin{proposition}\label{pr:3opt}
For all $K>0$, $g(K)\ge\max\{\alpha r,\,2\sqrt{\alpha\beta}\}$. For $K_0=\sqrt{\beta/\alpha}$,
\[
g(K_0)=2\sqrt{\alpha\beta}+2\alpha r\ \le\ 3\min_{K>0}g(K).
\]
\end{proposition}

\begin{proof}
First, $g(K)\ge\alpha(K+r)\ge\alpha r$. Second, $g(K)\ge(\alpha+\beta/K^2)K=\alpha K+\beta/K\ge2\sqrt{\alpha\beta}$. Substituting $K_0$ gives $\alpha+\beta/K_0^2=2\alpha$ and $g(K_0)=2\alpha(K_0+r)$.
\end{proof}

Thus the default block length $K\asymp\sqrt{\kappa/m}$ minimizes the upper bound $g$ up to a factor of three, whatever the relative cost of the two oracles; the guarantee concerns the bound, not the actual cost on a given problem. If $\kappa<m$, then $K_0<1$ and one takes $K=1$.

\section{Strong concavity: a gradient count independent of $m$}\label{sec:strong}

Under \cref{ass:cross}, one can reverse the roles of the blocks: accelerate on the primal function $P$ and use the small block only to build inexact models of it. The key inequality is a two-sided model of $P$. If $y_i$ is a $\delta$-maximizer of $f(v_i,\cdot)$, then for all $x\in\cX$
\begin{equation}\label{eq:model}
f(x,y_i)\ \le\ P(x)\ \le\ f(x,y_i)+\gamma_\chi\nrm{x-v_i}^2+2\delta,\qquad \gamma_\chi=\ell^2/\nu=\mu\chi,
\end{equation}
so $P$ admits a smooth inexact model in the sense of \citet{devolder2014first,stonyakin2021inexact}. An outer accelerated method with parameter $\Gamma=\mu+\gamma_\chi$ needs $O(\sqrt{1+\chi})$ steps to halve the error. Each step solves a proximal subproblem with condition number $1+L/\Gamma$, which costs $O(\sqrt{1+\kappa/(1+\chi)})$ gradients. The product, $O(\sqrt{\kappa+\chi})$, does not depend on $m$. The subproblems are \emph{not} solved to a prescribed accuracy: the inner run of \cref{alg:block} returns two points that satisfy an exact inequality, which is plugged directly into the outer potential. The method is stated as \cref{alg:cross} in \cref{app:minmax-cross}, together with the proof of the following theorem.

\begin{theorem}\label{thm:cross}
Let \cref{ass:minmax,ass:cross} hold, and let $0<\eps<B$. Put $\bar\chi=1+\chi$ and $\Lambda=1+\log\frac{B\bar\chi}{\eps}$. \Cref{alg:cross}, with the problems in $y$ solved by the center-of-gravity method, returns $(\hat x,\hat y)$ with $\Gap(\hat x,\hat y)\le\eps$ after
\[
\Tx=I_0+O\big(\sqrt{\kappa+\chi}\;\Lambda\big),
\qquad
\Ty=O\big(m\sqrt{1+\chi}\;\Lambda^2\big)
\]
oracle calls. If the problems in $y$ are solved by another method that returns a $\delta$-maximizer after $N_y(\delta)$ calls, then $\Ty=O\big(\sqrt{1+\chi}\,\Lambda\cdot N_y(\eps/(C\bar\chi^{5/2}))\big)$ for an absolute constant $C$.
\end{theorem}

\begin{corollary}\label{cor:combined}
Under the assumptions of \cref{thm:cross}, running the better of \cref{alg:ct,alg:cross} gives
\[
\Tx+\Ty=I_0+O\Big(\min\Big\{(m+\sqrt{m\kappa})\,\Lambda,\ \ \sqrt\kappa\,\Lambda+m\sqrt{1+\chi}\,\Lambda^2\Big\}\Big).
\]
In particular, if $\chi=O(1)$, then $\Tx=I_0+O(\sqrt\kappa\,\Lambda)$ and $\Ty=O(m\Lambda^2)$, so the complexity is additive up to one logarithm in the $y$-term. Up to logarithms, \cref{alg:cross} is preferable to \cref{alg:ct} exactly when $\chi\lesssim\kappa/m$. If only $\Tx$ counts, which is natural when a gradient in $x$ is a backward pass through a model, \cref{alg:cross} is preferable whenever $\chi\lesssim m\kappa$.
\end{corollary}

\paragraph{What the two-point method buys.}
A gradient count independent of $m$ alone is not new: by \eqref{eq:model}, an accelerated method on $P$ with the inexact gradient $\nabla_xf(v,\tilde y)$, where $\tilde y$ is a $\delta$-maximizer of $f(v,\cdot)$, is an inexact first-order method in the sense of \citet{devolder2014first} and also gives $\Tx=O(\sqrt{\kappa+\chi}\,\Lambda)$ up to logarithms, but it solves a problem in $y$ at each of its $O(\sqrt{\kappa+\chi})$ steps per halving of the error. \Cref{alg:cross} solves only $O(\sqrt{1+\chi})$ of them (\cref{tab:strong}); $\sqrt{\kappa+\chi}$ is the $x$-part of the lower bound of \citet{zhang2022lower}. The independence of $m$ is for a fixed $\chi$, which in natural examples may grow with $m$ through the Jacobian of the group losses.

\begin{table}[t]
\centering
\caption{Oracle calls up to constants and the initialization term $I_0$, with $\Lambda$ as in \cref{thm:cross}. Second row: reverse nested scheme, up to logarithms.}\label{tab:strong}
\small
\begin{tabular}{lcc}
\toprule
Method & $\Tx$ & $\Ty$ \\
\midrule
Certificate transport (\cref{thm:general}) & $(m+\sqrt{m\kappa})\,\Lambda$ & $m\Lambda$ \\
Accelerated method on $P$ + cutting planes in $y$ & $\sqrt{\kappa+\chi}\,\Lambda$ & $m\sqrt{\kappa+\chi}\,\Lambda^2$ \\
Two-point method (\cref{alg:cross}, \cref{thm:cross}) & $\sqrt{\kappa+\chi}\,\Lambda$ & $m\sqrt{1+\chi}\,\Lambda^2$ \\
\bottomrule
\end{tabular}
\end{table}

\section{Experiments}\label{sec:experiments}

The analysis predicts where certificate transport helps. A nested scheme pays for every inner solve, and even with warm starts this cost grows with $\kappa$ and with the distance between consecutive slice solutions; CT pays a fixed $K_0+1\approx\sqrt{8\kappa/m}$ gradients per query. The advantage should therefore appear with \emph{many dual variables} and \emph{weak regularization}.

\paragraph{Setup.}
We use group DRO, learning with a few constraints and Neyman--Pearson classification on Adult and 20~Newsgroups, and synthetic least-squares group DRO; $f(x,y)=\sum_jw_j(y)L_j(x)+\frac\mu2\nrm x^2+\psi(y)$ with $w$ affine in $y$, and the ridge parameter sets $\kappa$. \emph{CT} is \cref{alg:ct} with the default $K_0$ and three changes that keep \cref{thm:general} valid (\cref{app:exp-methods}): the largest step sizes allowed by the proof of \cref{lem:block}, $\sigma=\mu$, and a direct certificate on the queried slice (one extra gradient per query). Baselines: the textbook nested scheme \eqref{eq:nested} (cold start, inner accuracy $\eps/2$); warm-started inner runs; warm starts with inner accuracy proportional to the certified width; and, the strongest competitor, the latter with the localizer and epochs of \cref{alg:ct} (``level'', an inexact level method \citep{lemarechal1995new}), whose fraction $c$ is tuned over $\{0.03,0.1,0.3,0.5,1\}$ in hindsight. Extragradient and GDA with multiplicative weights are tuned on a grid. Reported gaps are computed independently and rigorously; the true gap never exceeded the certified width (\cref{app:experiments}).

\begin{figure}[t]
\centering
\includegraphics[width=\linewidth]{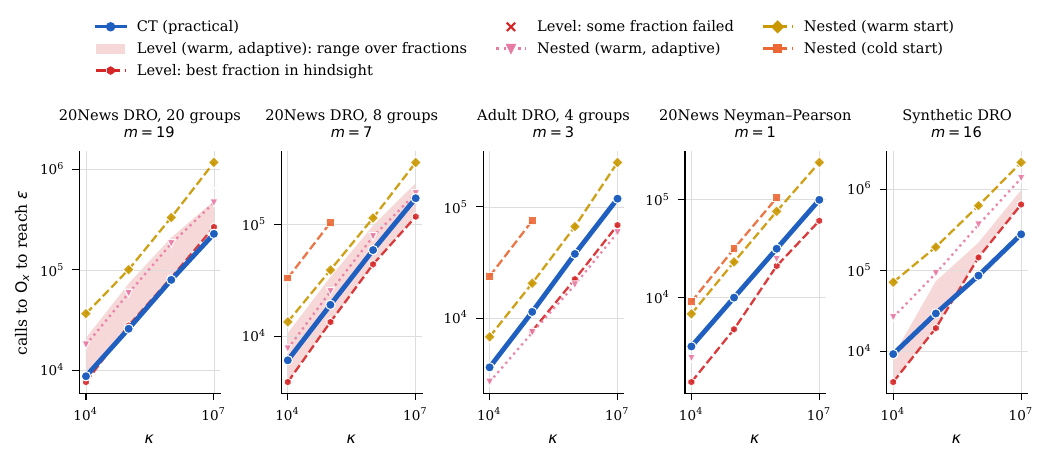}
\caption{Calls to $\Ox$ to certify $\Gap\le10^{-6}$ as $\kappa$ grows, panels ordered by decreasing $m$. Band: the level baseline over its fractions $c$; dashed: its best $c$ for each $\kappa$. CT has no tuned parameter.}
\label{fig:weak}
\end{figure}

\paragraph{Many groups, weak regularization (\cref{fig:weak}).}
For $m\ge16$ and $\kappa\ge10^6$ CT beats every configuration of the nested and level baselines (\cref{tab:weak,fig:weakconv}). On the 20-group problem at $\kappa=10^7$ it needs $228$k calls to $\Ox$ and $88$ to $\Oy$; the level baseline needs $267$k and $1{,}082$ with its best fraction and $402$k with the default one, and the nested schemes $465$k--$1.2$M. On the synthetic instance with $m=16$ the numbers are $278$k for CT and $647$k for the best level fraction. For smaller $\kappa$ or $m\le7$ the tuned level baseline is faster, by $1.3$--$2.3\times$: CT's cost per query decays like $1/\sqrt m$, that of a warm-started solve does not. No fixed fraction is both fast and reliable (\cref{tab:levelfrac}): relative to CT the geometric-mean ratio of $\Tx$ is $2.0,1.2,1.1,0.91,1.07$ for $c=0.03,\dots,1$, but $c=0.5$ failed to certify the target once and $c=1$ in $8$ of $17$ cells.

\begin{figure}[t]
\centering
\includegraphics[width=\linewidth]{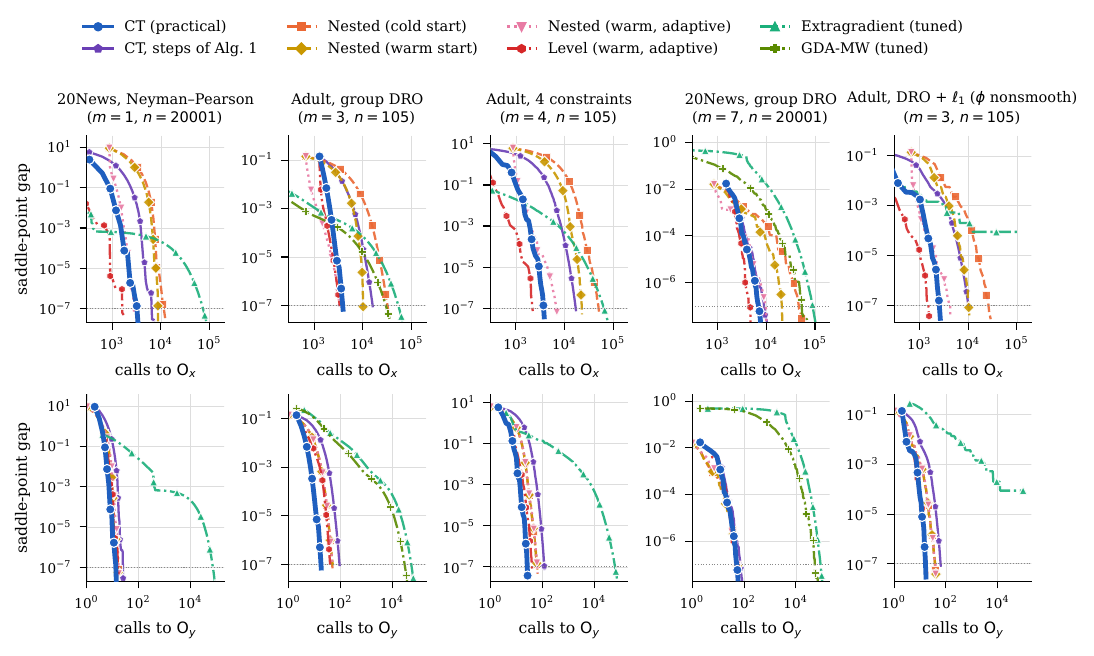}
\caption{Gap versus calls to $\Ox$ (top) and $\Oy$ (bottom) on five real problems, $\kappa=10^4$, target $\eps=10^{-7}$ (dotted). The last problem has an $\ell_1$ penalty on the group weights, which makes $\phi$ nonsmooth.}
\label{fig:real}
\end{figure}

\paragraph{Strong regularization, few groups (\cref{fig:real}).}
At $\kappa=10^4$ CT needs $2.7$--$7.0$k gradients, $3.3$--$13\times$ fewer than the textbook nested scheme and $2.5$--$6\times$ fewer than the warm-started one; the tuned level baseline is $1.2$--$2.2\times$ faster. Single-loop methods need $8$--$21\times$ more gradients and over $10^3\times$ more $y$-queries; on the nonsmooth $\ell_1$ problem extragradient does not reach $10^{-5}$ in $10^5$ calls. The second logarithm of \eqref{eq:nested} is visible directly: the cost of CT per digit of accuracy is constant, while for the textbook scheme it grows from $1.6$k to $3.6$k (\cref{fig:eps}); in $m$ the fitted exponent of $\Tx$ is $0.61$ for CT versus $1.16$ for the nested scheme (\cref{fig:scaling}).

\begin{figure}[t]
\centering
\includegraphics[width=\linewidth]{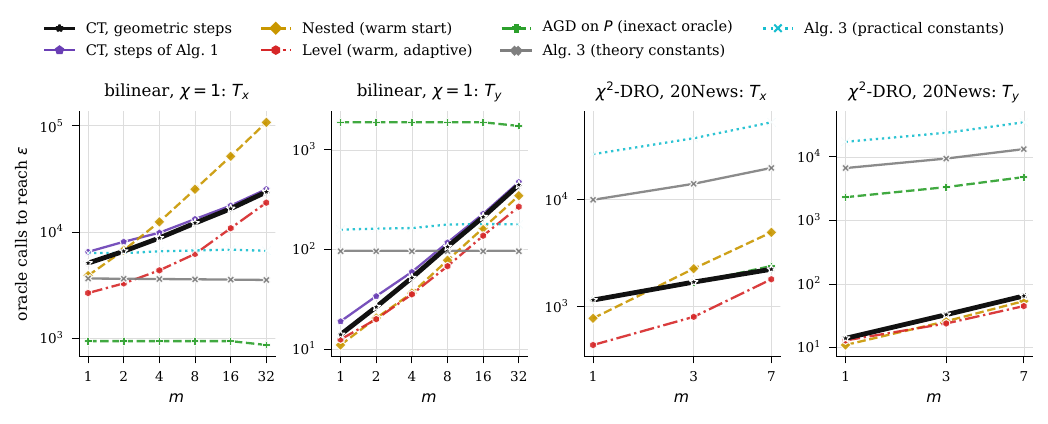}
\caption{Strongly concave problems, $\eps=10^{-6}$. Left: bilinear problem with coupling $\chi=1$ independent of $m$. Right: $\chi^2$-regularized group DRO on 20~Newsgroups, where the certified $\chi=1.5$--$6.2\cdot10^4$ grows with the number of groups. $\Ty$ counts $y$-subproblems.}
\label{fig:strong}
\end{figure}

\paragraph{Strong concavity (\cref{fig:strong}).}
On the bilinear problem \cref{alg:cross} needs $3.6$k calls to $\Ox$ for every $m\le32$, as \cref{thm:cross} predicts, while CT grows as $m^{0.39}$. The independence of $m$ is not specific to \cref{alg:cross}: AGD on $P$ with inexact gradients needs $0.9$k calls. What \cref{alg:cross} saves is $y$-subproblems, $96$ versus $1{,}900$, in line with $\sqrt{1+\chi}$ versus $\sqrt{\kappa+\chi}$ per halving of the error. On $\chi^2$-DRO the large $\chi$ makes \cref{alg:cross} cost $10$--$53$k calls, while CT, which ignores strong concavity, needs $1.2$--$2.2$k, as few as AGD on $P$ with $70$--$170\times$ fewer $y$-queries.

\takeaway{With many dual variables and weak regularization CT is the fastest method, by up to $2.3\times$ in gradients and $12\times$ in $y$-queries over the best-tuned heuristic and $2$--$8\times$ over nested schemes, with no tuning and a certified gap. With few dual variables and strong regularization a tuned warm-started level method is faster.}

\section{Conclusion}\label{sec:conclusion}

Certificate transport removes the multiplicative cost of nested schemes while asking nothing of the small block beyond concavity and Lipschitz continuity: an inexact inner solution either cuts the localizer or transports the lower model. The product $m\sqrt\kappa$ becomes $\sqrt{m\kappa}$, the second logarithm disappears, and for $m=1$ both counts are optimal; under strong concavity, a two-point method makes the gradient count independent of $m$.

\paragraph{What the certificate gives.}
For group DRO, $a-b\le\eps$ certifies the regularized worst-group risk $P(\bar x)=\max_jL_j(\bar x)+\frac\mu2\nrm{\bar x}^2$ up to $\eps$; for constraints with multipliers in $[0,\bar y]^m$, it certifies the problem with the exact penalty $\bar y\sum_j\max\{0,g_j(x)\}$ on the violations $g_j(x)$.

\paragraph{Limitations.}
We assume exact oracles, full gradients in $x$ and known $L,\mu$, and treat centers of gravity as free; for larger $m$ they must be approximated (\cref{rem:implementation}), as exact computation is hard \citep{rademacher2007approximating}. If $\Oy$ returns values with error $\delta_y$, shifting the planes up and the transported value down by $\delta_y$ keeps all certificates valid, with attainable width of order $\delta_y$. Optimality of $\sqrt{m\kappa}$ for $m>1$ is open; the separate lower bounds give $\Omega((\sqrt\kappa+m)\log\frac1\eps)$.

\clearpage
\appendix
\section{Tools: accelerated blocks, initialization, centers of gravity, and recovery}\label{app:toolbox}

Throughout the appendices, $\cX\subseteq\RR^n$ is closed and convex, $\proj_\cX$ is the Euclidean projection onto $\cX$, and $\cY$ is a convex body. All statements about a function $F$ on $\cX$ use only its values and gradients at points of $\cX$.

We will repeatedly use the following property of the projection: if $p\in\RR^n$ and $z=\proj_\cX(p)$, then $\ip{p-z}{x-z}\le0$ for all $x\in\cX$, and therefore
\begin{equation}\label{eq:projection}
\nrm{x-p}^2=\nrm{x-z}^2+2\ip{x-z}{z-p}+\nrm{z-p}^2\ \ge\ \nrm{x-z}^2+\nrm{z-p}^2\qquad\forall x\in\cX .
\end{equation}

\subsection{An accelerated block with an arbitrary quadratic prior}

The block of \cref{alg:block} is an estimate-sequence method in which the usual starting quadratic is replaced by an arbitrary quadratic \emph{prior} $\pi$. The prior enters the analysis only through the inequality $\pi+A_kF\ge c+A_kF(u_k)+\ldots$ below; in \cref{alg:ct} it is the lower model of a \emph{different} slice $f(\cdot,y_b)$.

\paragraph{Setting.}
Let $F:\cX\to\RR$ be $\mu_F$-strongly convex and $L_F$-smooth on $\cX$ with $0\le\mu_F\le L_F$, and let
\[
\pi(x)=c+\frac\gamma2\nrm{x-\bar x}^2,\qquad c\in\RR,\ \ \gamma>0,\ \ \bar x\in\cX .
\]
Every gradient $g=\nabla F(t)$ at a point $t\in\cX$ gives a quadratic lower bound on $F$ by strong convexity:
\begin{equation}\label{eq:omega}
F(x)\ \ge\ \omega_t(x)\defeq F(t)+\ip{g}{x-t}+\frac{\mu_F}2\nrm{x-t}^2\qquad\forall x\in\cX .
\end{equation}
Given nonnegative weights $a_1,a_2,\ldots$ and query points $t_1,t_2,\ldots\in\cX$, put $A_0=0$, $A_i=A_{i-1}+a_i$, and
\[
\Psi_i(x)\defeq\pi(x)+\sum_{j=1}^ia_j\,\omega_{t_j}(x),\qquad \Psi_0=\pi .
\]
Summing \eqref{eq:omega} over $j\le i$ with weights $a_j\ge0$ and adding $\pi$ gives, for every $i$,
\begin{equation}\label{eq:psiupper}
\Psi_i(x)\ \le\ \pi(x)+A_iF(x)\qquad\forall x\in\cX .
\end{equation}

\paragraph{The iteration.}
Start from $u_0=z_0=\bar x$ and, for $i=1,\ldots,k$, compute
\begin{equation}\label{eq:blockiter}
\begin{gathered}
a_i=\frac{(i+1)\gamma}{2L_F},\qquad A_i=A_{i-1}+a_i,\qquad t_i=\frac{A_{i-1}u_{i-1}+a_iz_{i-1}}{A_i},\qquad g_i=\nabla F(t_i),\\
z_i=\argmin_{x\in\cX}\Psi_i(x),\qquad u_i=\frac{A_{i-1}u_{i-1}+a_iz_i}{A_i}.
\end{gathered}
\end{equation}
The points $t_i$ are the points $x_i$ of \cref{alg:block}; we rename them to keep $x$ free for a generic point of $\cX$. By induction, $t_i,u_i,z_i\in\cX$: $z_i\in\cX$ by definition, and $t_i,u_i$ are convex combinations of points of $\cX$.

\begin{lemma}[accelerated block]\label{lem:block}
In the setting above, the following hold.
\begin{enumerate}[label=\textup{(\roman*)},leftmargin=2em,itemsep=1pt]
\item $z_i=\proj_\cX(r_i/\Gamma_i)$, where $\Gamma_i=\gamma+\mu_FA_i$ and $r_i=\gamma\bar x+\sum_{j\le i}a_j(\mu_Ft_j-g_j)$. Hence each step uses exactly one gradient and no function values.
\item $A_k=\gamma k(k+3)/(4L_F)$.
\item For all $x\in\cX$,
\begin{equation}\label{eq:blockineq}
\pi(x)+A_kF(x)\ \ge\ \Psi_k(x)\ \ge\ c+A_kF(u_k)+\frac{\gamma+\mu_FA_k}{2}\nrm{x-z_k}^2 .
\end{equation}
\end{enumerate}
\end{lemma}

\begin{proof}
\emph{Step 1: $\Psi_i$ is an isotropic quadratic.}
Expanding the squares in $\pi$ and in each $\omega_{t_j}$ and collecting the terms of degree two, one and zero in $x$,
\[
\begin{aligned}
\Psi_i(x)
&=c+\frac\gamma2\nrm{x-\bar x}^2+\sum_{j\le i}a_j\Big[F(t_j)+\ip{g_j}{x-t_j}+\frac{\mu_F}2\nrm{x-t_j}^2\Big]\\
&=\frac{\gamma+\mu_FA_i}{2}\nrm{x}^2-\Big\langle\gamma\bar x+\sum_{j\le i}a_j(\mu_Ft_j-g_j),\,x\Big\rangle+\mathrm{const}_i
=\frac{\Gamma_i}2\nrm{x}^2-\ip{r_i}{x}+\mathrm{const}_i ,
\end{aligned}
\]
where $\mathrm{const}_i$ does not depend on $x$ (it contains the unknown values $F(t_j)$). Completing the square,
\begin{equation}\label{eq:psisquare}
\Psi_i(x)=\Psi_i^\star+\frac{\Gamma_i}2\Big\|x-\frac{r_i}{\Gamma_i}\Big\|^2
\end{equation}
for some constant $\Psi_i^\star$. The minimizer of $\Psi_i$ over $\cX$ is therefore the projection of $r_i/\Gamma_i$ onto $\cX$, which proves (i). For $i=0$ we get $r_0/\Gamma_0=\bar x\in\cX$, so $z_0=\bar x$, in agreement with the initialization.

\emph{Step 2: quadratic growth around $z_i$.}
Apply \eqref{eq:projection} with $p=r_i/\Gamma_i$ and $z=z_i$ to \eqref{eq:psisquare}:
\begin{equation}\label{eq:psigrowth}
\Psi_i(x)\ \ge\ \Psi_i(z_i)+\frac{\Gamma_i}2\nrm{x-z_i}^2=\min_\cX\Psi_i+\frac{\Gamma_i}2\nrm{x-z_i}^2\qquad\forall x\in\cX .
\end{equation}

\emph{Step 3: the weights.}
Since $\sum_{j=1}^i(j+1)=\frac{i(i+1)}2+i=\frac{i(i+3)}2$, we have $A_i=\frac{\gamma}{2L_F}\cdot\frac{i(i+3)}2=\frac{\gamma i(i+3)}{4L_F}$, which proves (ii). Moreover, for $i\ge1$,
\begin{equation}\label{eq:stepcond}
L_Fa_i^2=\frac{\gamma^2(i+1)^2}{4L_F}\ \le\ \frac{\gamma^2i(i+3)}{4L_F}=\gamma A_i\ \le\ \Gamma_{i-1}A_i ,
\end{equation}
because $(i+1)^2=i^2+2i+1\le i^2+3i=i(i+3)$ when $i\ge1$, and $\Gamma_{i-1}=\gamma+\mu_FA_{i-1}\ge\gamma$.

\emph{Step 4: the invariant $c+A_iF(u_i)\le\min_\cX\Psi_i$.}
We argue by induction on $i$. For $i=0$, both sides equal $c$, since $A_0=0$ and $\min_\cX\Psi_0=\pi(\bar x)=c$. Assume the invariant for $i-1$ and evaluate $\Psi_i=\Psi_{i-1}+a_i\omega_{t_i}$ at its minimizer $z_i$:
\[
\min_\cX\Psi_i=\Psi_i(z_i)=\Psi_{i-1}(z_i)+a_i\Big[F(t_i)+\ip{g_i}{z_i-t_i}+\frac{\mu_F}2\nrm{z_i-t_i}^2\Big].
\]
We bound the two terms on the right. By \eqref{eq:psigrowth} for $i-1$ and then the induction hypothesis,
\[
\Psi_{i-1}(z_i)\ \ge\ \min_\cX\Psi_{i-1}+\frac{\Gamma_{i-1}}2\nrm{z_i-z_{i-1}}^2\ \ge\ c+A_{i-1}F(u_{i-1})+\frac{\Gamma_{i-1}}2\nrm{z_i-z_{i-1}}^2 .
\]
In the bracket we drop the nonnegative term $\frac{\mu_F}2\nrm{z_i-t_i}^2$ (recall $a_i\ge0$). Hence
\[
\min_\cX\Psi_i\ \ge\ c+A_{i-1}F(u_{i-1})+a_i\big[F(t_i)+\ip{g_i}{z_i-t_i}\big]+\frac{\Gamma_{i-1}}2\nrm{z_i-z_{i-1}}^2 .
\]
Next, by convexity of $F$, $F(u_{i-1})\ge F(t_i)+\ip{g_i}{u_{i-1}-t_i}$. Substituting and using $A_{i-1}+a_i=A_i$,
\[
\min_\cX\Psi_i\ \ge\ c+A_iF(t_i)+\big\langle g_i,\ A_{i-1}u_{i-1}+a_iz_i-A_it_i\big\rangle+\frac{\Gamma_{i-1}}2\nrm{z_i-z_{i-1}}^2 .
\]
By the definition of $u_i$, $A_{i-1}u_{i-1}+a_iz_i=A_iu_i$, so the inner product equals $A_i\ip{g_i}{u_i-t_i}$ and
\begin{equation}\label{eq:inv-mid}
\min_\cX\Psi_i\ \ge\ c+A_i\big[F(t_i)+\ip{g_i}{u_i-t_i}\big]+\frac{\Gamma_{i-1}}2\nrm{z_i-z_{i-1}}^2 .
\end{equation}
Now we use smoothness. Subtracting the definitions of $t_i$ and $u_i$ gives $u_i-t_i=\frac{a_i}{A_i}(z_i-z_{i-1})$. By $L_F$-smoothness,
\[
F(u_i)\ \le\ F(t_i)+\ip{g_i}{u_i-t_i}+\frac{L_F}2\nrm{u_i-t_i}^2=F(t_i)+\ip{g_i}{u_i-t_i}+\frac{L_Fa_i^2}{2A_i^2}\nrm{z_i-z_{i-1}}^2 .
\]
Multiplying by $A_i$ and inserting the result into \eqref{eq:inv-mid},
\[
\min_\cX\Psi_i\ \ge\ c+A_iF(u_i)+\frac12\Big[\Gamma_{i-1}-\frac{L_Fa_i^2}{A_i}\Big]\nrm{z_i-z_{i-1}}^2\ \ge\ c+A_iF(u_i),
\]
where the bracket is nonnegative by \eqref{eq:stepcond}. This completes the induction.

\emph{Step 5: conclusion.}
The first inequality in \eqref{eq:blockineq} is \eqref{eq:psiupper} with $i=k$. For the second, combine \eqref{eq:psigrowth} for $i=k$ with the invariant of Step~4: $\Psi_k(x)\ge\min_\cX\Psi_k+\frac{\Gamma_k}2\nrm{x-z_k}^2\ge c+A_kF(u_k)+\frac{\Gamma_k}2\nrm{x-z_k}^2$, and $\Gamma_k=\gamma+\mu_FA_k$.
\end{proof}

\Cref{alg:block} is \cref{lem:block} with $F=f(\cdot,y)$, $\mu_F=\mu$, $L_F=L$, and prior $\pi=q_b$, that is, $c=b$, $\gamma=\sigma=\mu/2$ and $\bar x=x_b$. In this case
\[
A_K=\frac{\sigma K(K+3)}{4L}=\frac{K(K+3)}{8\kappa}.
\]

\subsection{Certificate transport}

\begin{corollary}[transport]\label{cor:transport}
Let \cref{ass:minmax} hold, and let $(b,x_b,y_b)$ be a lower model, that is, $x_b\in\cX$, $y_b\in\cY$, and
\[
f(x,y_b)\ \ge\ q_b(x)=b+\frac\sigma2\nrm{x-x_b}^2\quad\forall x\in\cX,\qquad \sigma=\mu/2 .
\]
Let $y_c\in\cY$, let $(u,z,A)=\textsc{Block}(f(\cdot,y_c),q_b,K)$, and let $v=f(u,y_c)$. Put $\vartheta=A/(1+A)$. Then the triple
\[
b^+=\frac{b+Av}{1+A},\qquad x_b^+=z,\qquad y_b^+=(1-\vartheta)y_b+\vartheta y_c=\frac{y_b+Ay_c}{1+A}
\]
is again a lower model: $x_b^+\in\cX$, $y_b^+\in\cY$, and $f(x,y_b^+)\ge b^++\frac\sigma2\nrm{x-x_b^+}^2$ for all $x\in\cX$. This is inequality \eqref{eq:transport}.
\end{corollary}

\begin{proof}
First, $z\in\cX$ by \cref{lem:block}(i), and $y_b^+\in\cY$ as a convex combination of $y_b,y_c\in\cY$. Fix $x\in\cX$ and apply \eqref{eq:blockineq} to $F=f(\cdot,y_c)$ with prior $\pi=q_b$ (so $c=b$, $\gamma=\sigma$, $\mu_F=\mu$) and $F(u)=v$:
\[
q_b(x)+A\,f(x,y_c)\ \ge\ b+Av+\frac{\sigma+\mu A}{2}\nrm{x-z}^2 .
\]
Since $f(x,y_b)\ge q_b(x)$, the left-hand side is at most $f(x,y_b)+Af(x,y_c)$. By concavity of $f(x,\cdot)$,
\[
f(x,y_b^+)\ \ge\ (1-\vartheta)f(x,y_b)+\vartheta f(x,y_c)=\frac{f(x,y_b)+Af(x,y_c)}{1+A}.
\]
Combining the two displays and dividing by $1+A$,
\[
f(x,y_b^+)\ \ge\ \frac{b+Av}{1+A}+\frac{\sigma+\mu A}{2(1+A)}\nrm{x-z}^2\ \ge\ b^++\frac\sigma2\nrm{x-z}^2 ,
\]
where the last step uses $\frac{\sigma+\mu A}{1+A}\ge\sigma$, which holds because $\mu\ge\sigma$.
\end{proof}

The conclusion holds for \emph{every} value of $v$: the transported triple is always a valid lower model. The comparison of $v$ with the threshold in \cref{alg:ct} matters only for progress, because $b^+-b=\vartheta(v-b)$ is positive exactly when $v>b$. The proof also shows where the curvature $\sigma=\mu/2$ of the lower model comes from: the transported curvature is $\frac{\sigma+\mu A}{1+A}\ge\sigma$ for every $\sigma\le\mu$, so transport works equally with $\sigma=\mu$, and the factor $\frac12$ is needed only by the initialization of \cref{lem:init}.

\subsection{Initialization of the lower model}

\begin{lemma}[initialization]\label{lem:init}
Let $F$ be $\mu$-strongly convex and $L$-smooth on $\cX$, let $\diam\cX\le\Rx$, let $x^\star=\argmin_\cX F$, and let $E_{\rm target}>0$. Using at most
\[
\lceil\sqrt{8\kappa}\,\rceil\Big(2+\Big\lceil\log_2^+\frac{\mu\Rx^2}{4E_{\rm target}}\Big\rceil\Big)
\]
gradients of $F$ and no function values, one can find $u,z\in\cX$ such that
\[
F(u)-F(x^\star)\le\frac{E_{\rm target}}2
\qquad\text{and}\qquad
F(x)\ \ge\ F(u)-3E_{\rm target}+\frac\sigma2\nrm{x-z}^2\quad\forall x\in\cX,
\]
where $\sigma=\mu/2$.
\end{lemma}

\begin{proof}
\emph{One run.}
Fix $v\in\cX$ and apply \cref{lem:block} with $\mu_F=\mu$, $L_F=L$, prior $\pi(x)=\frac\gamma2\nrm{x-v}^2$ (so $c=0$, $\bar x=v$), $\gamma=\mu/2$, and $k=\lceil\sqrt{8\kappa}\,\rceil$ steps. Let $u,z$ be its output and
\[
A\defeq A_k=\frac{\mu k(k+3)}{8L}\ \ge\ \frac{k^2}{8\kappa}\ \ge\ 1 .
\]
Dividing \eqref{eq:blockineq} by $A$ gives
\begin{equation}\label{eq:initstep}
F(x)\ \ge\ F(u)+\frac{\mu/2+\mu A}{2A}\nrm{x-z}^2-\frac{\mu}{4A}\nrm{x-v}^2\qquad\forall x\in\cX .
\end{equation}
We extract three consequences.

(a) \emph{Accuracy of $u$.} Take $x=x^\star$ in \eqref{eq:initstep} and drop the nonnegative term with $\nrm{x^\star-z}^2$: $F(u)-F(x^\star)\le\frac{\mu}{4A}\nrm{x^\star-v}^2$.

(b) \emph{Position of $z$.} Take $x=x^\star$ again and use $F(x^\star)\le F(u)$:
$\frac{\mu/2+\mu A}{2A}\nrm{x^\star-z}^2\le\frac{\mu}{4A}\nrm{x^\star-v}^2$, that is, $\nrm{x^\star-z}^2\le\frac{\nrm{x^\star-v}^2}{1+2A}\le\frac{\nrm{x^\star-v}^2}3$. By the triangle inequality,
\[
\nrm{z-v}^2\le\big(\nrm{z-x^\star}+\nrm{x^\star-v}\big)^2\le\big(1+3^{-1/2}\big)^2\nrm{x^\star-v}^2\le3\nrm{x^\star-v}^2 .
\]

(c) \emph{Lower model centered at $z$.} Insert $\nrm{x-v}^2\le2\nrm{x-z}^2+2\nrm{z-v}^2$ into \eqref{eq:initstep}. The coefficient of $\nrm{x-z}^2$ becomes $\frac{\mu/2+\mu A}{2A}-\frac{\mu}{2A}=\frac\mu2-\frac{\mu}{4A}\ge\frac\mu4=\frac\sigma2$ (since $A\ge1$), and by (b) the remaining term is at least $-\frac{\mu}{2A}\cdot3\nrm{x^\star-v}^2\ge-\frac{3\mu}2\nrm{x^\star-v}^2$. Thus
\[
F(x)\ \ge\ F(u)+\frac\sigma2\nrm{x-z}^2-\frac{3\mu}2\nrm{x^\star-v}^2\qquad\forall x\in\cX .
\]

Suppose now that $F(v)-F(x^\star)\le E$. Since $x^\star$ minimizes the $\mu$-strongly convex $F$ over $\cX$, we have $F(v)-F(x^\star)\ge\frac\mu2\nrm{v-x^\star}^2$, hence $\nrm{x^\star-v}^2\le2E/\mu$. Then (a) gives $F(u)-F(x^\star)\le\frac{E}{2A}\le\frac E2$, and (c) gives $F(x)\ge F(u)-3E+\frac\sigma2\nrm{x-z}^2$. In words: \emph{one run started from a point with accuracy $E$ halves the accuracy and produces a lower model with constant $F(u)-3E$.}

\emph{Restarts.}
Start from an arbitrary $v_0\in\cX$. Since $\nrm{x^\star-v_0}\le\Rx$, (a) shows that the first run returns $u_1$ with $F(u_1)-F(x^\star)\le\frac{\mu\Rx^2}{4}\eqqcolon E_1$. Restarting from $u_s$ gives $E_{s+1}=E_s/2$, so $E_s=2^{1-s}E_1\le E_{\rm target}$ as soon as $s\ge1+\log_2\frac{\mu\Rx^2}{4E_{\rm target}}$. After $s=1+\lceil\log_2^+\frac{\mu\Rx^2}{4E_{\rm target}}\rceil$ runs we have a point $v$ with $F(v)-F(x^\star)\le E_{\rm target}$, and one more run from $v$ gives the claim with $E=E_{\rm target}$. Each run costs $k=\lceil\sqrt{8\kappa}\,\rceil$ gradients, and no run uses function values.
\end{proof}

The lemma produces a lower model whose constant $F(u)-3E_{\rm target}$ involves the value $F(u)$, which the lemma itself does not compute. In \cref{alg:ct} this value comes for free with the first query in $y$.

\begin{corollary}[start of \cref{alg:ct}]\label{cor:start}
Let \cref{ass:minmax} hold, let $y_0\in\cY$ be arbitrary, and apply \cref{lem:init} to $F=f(\cdot,y_0)$ with $E_{\rm target}=B/3$. Make one call $(v_0,g_0)=\Oy(u,y_0)$, so that $v_0=f(u,y_0)$. Then
\[
b=v_0-B,\qquad x_b=z,\qquad y_b=y_0
\]
is a lower model, the plane $\ell_0(y)=v_0+\ip{g_0}{y-y_0}$ satisfies $\ell_0\ge f(u,\cdot)$ on $\cY$, and $a=\max_\cY\ell_0\le v_0+B$. Hence $a-b\le2B$. The cost is
\[
I_0=\lceil\sqrt{8\kappa}\,\rceil\Big(2+\Big\lceil\log_2^+\frac{3\mu\Rx^2}{4B}\Big\rceil\Big)=O\Big(\sqrt\kappa\Big[1+\logp\frac{\mu\Rx^2}{B}\Big]\Big)
\]
calls to $\Ox$ and one call to $\Oy$.
\end{corollary}

\begin{proof}
By \cref{lem:init} with $E_{\rm target}=B/3$, $f(x,y_0)\ge v_0-B+\frac\sigma2\nrm{x-z}^2$ for all $x\in\cX$, which is the lower model. The plane is an upper bound on $f(u,\cdot)$ by concavity. For every $y\in\cY$, $\ell_0(y)\le v_0+\nrm{g_0}\nrm{y-y_0}\le v_0+\My\Ry=v_0+B$.
\end{proof}

\subsection{Centers of gravity}

\begin{lemma}[Grünbaum]\label{lem:grunbaum}
Let $S\subset\RR^m$ be a convex body with center of gravity $c$, and let $H$ be a closed half-space with $c\notin\operatorname{int}H$. Then $\vol(S\cap H)\le(1-e^{-1})\vol(S)$ \citep{grunbaum1960partitions}. For $m=1$ the constant $1-e^{-1}$ can be replaced by $1/2$.
\end{lemma}

In our use, $H=\{y:\ip{g}{y-c}\ge0\}$ is a half-space whose boundary passes through $c$, so $c\notin\operatorname{int}H$.

The next lemma is the center-of-gravity method with inexact function values; it is used in \cref{app:minmax-cross} to solve the problems in $y$.

\begin{lemma}[center of gravity with inexact values]\label{lem:cog}
Let $\psi$ be concave on $\cY$. Put $\cY_0=\cY$ and, for $t=1,2,\ldots$, let the query $y_t$ be the center of gravity of $\cY_{t-1}$ and
\[
\cY_t=\{y\in\cY_{t-1}:\ \ip{g_t}{y-y_t}\ge0\}.
\]
Suppose that each query returns a plane $\ell_t(y)=w_t+\ip{g_t}{y-y_t}$ with $\ell_t\ge\psi$ on $\cY$ and $\nrm{g_t}\le\My$, and a number $\beta_t\le\psi(y_t)$ with $w_t-\beta_t\le\delta$. Let $a_t=\max_\cY\min_{s\le t}\ell_s$. Then
\[
a_t-\max_{s\le t}\beta_s\ \le\ B(1-e^{-1})^{t/m}+\delta .
\]
In particular, a $\delta$-maximizer of a concave $\My$-Lipschitz function on $\cY$ can be found with $O(m\log\frac B\delta)$ exact queries of its value and a supergradient.
\end{lemma}

\begin{proof}
\emph{Degenerate cases.} If $g_s=0$ for some $s\le t$, then $\ell_s\equiv w_s$, so $a_t\le\max_\cY\ell_s=w_s\le\beta_s+\delta$, and the claim holds. If $\vol\cY_{s}=0$ for some $s<t$, the argument below applied at time $s$ (with any $\lambda\in(0,1]$) gives $a_s-\max_{s'\le s}\beta_{s'}\le\delta$, and the claim holds at all later times because $a_t\le a_s$. So assume neither happens.

\emph{Volume.} Each $\cY_t$ is the intersection of $\cY_{t-1}$ with a half-space whose boundary passes through the center of gravity $y_t$ of $\cY_{t-1}$. By \cref{lem:grunbaum}, $\vol\cY_t\le(1-e^{-1})\vol\cY_{t-1}$, so $\vol\cY_t\le(1-e^{-1})^t\vol\cY$.

\emph{A point outside $\cY_t$ with a large value of the model.} Let $U=\min_{s\le t}\ell_s$, let $y^\circ\in\argmax_\cY U$, so that $U(y^\circ)=a_t$, and fix $\lambda\in\big((1-e^{-1})^{t/m},1\big]$. The set $S=y^\circ+\lambda(\cY-y^\circ)$ is contained in $\cY$ by convexity, and $\vol S=\lambda^m\vol\cY>\vol\cY_t$. Hence there is a point $y'\in S\setminus\cY_t$, and we can write $y'=y^\circ+\lambda(y''-y^\circ)$ with $y''\in\cY$. On the one hand, $U$ is concave and $\My$-Lipschitz (a minimum of affine functions with slopes of norm at most $\My$), and $U(y'')\ge U(y^\circ)-\My\nrm{y''-y^\circ}\ge a_t-B$. Therefore
\[
U(y')\ \ge\ (1-\lambda)U(y^\circ)+\lambda U(y'')\ \ge\ a_t-\lambda B .
\]
On the other hand, $y'\in\cY\setminus\cY_t$, so some cut excluded it: $\ip{g_s}{y'-y_s}<0$ for some $s\le t$. For this $s$,
\[
U(y')\ \le\ \ell_s(y')=w_s+\ip{g_s}{y'-y_s}\ <\ w_s\ \le\ \beta_s+\delta .
\]
Comparing the two bounds gives $a_t-\beta_s<\lambda B+\delta$. Letting $\lambda\downarrow(1-e^{-1})^{t/m}$ proves the claim.

\emph{Exact values.} If the queries are exact, take $w_t=\beta_t=\psi(y_t)$ and $\delta=0$, with $g_t$ a supergradient of $\psi$ at $y_t$, so that $\ell_t\ge\psi$ by concavity. Since $\max_\cY\psi\le a_t$, the best query point $y_s$ satisfies $\max_\cY\psi-\psi(y_s)\le B(1-e^{-1})^{t/m}$, which is at most $\delta$ after $t=\lceil m\log(B/\delta)/\log\frac{e}{e-1}\rceil$ queries.
\end{proof}

\begin{remark}[implementing centers of gravity]\label{rem:implementation}
In the oracle model, centers of gravity are free. If the exact center of gravity is replaced by an approximate one computed by a random walk \citep{bertsimas2004solving}, every central cut still removes a constant fraction of the volume with high probability, which changes only the absolute constants in \cref{thm:general}. In general, computing the center of gravity exactly is hard \citep{rademacher2007approximating}. For $m\le3$, exact centers of gravity of polytopes are cheap to compute, and for $m=1$ the center of gravity is the midpoint of a segment.
\end{remark}

\subsection{Recovery of a primal point}

\begin{lemma}[recovery]\label{lem:recovery}
Let $f(\cdot,y)$ be convex on $\cX$ for every $y\in\cY$. Let $u_1,\ldots,u_N\in\cX$, and let the planes $\ell_i(y)=v_i+\ip{g_i}{y-y_i}$ satisfy $\ell_i\ge f(u_i,\cdot)$ on $\cY$. Put $a=\max_\cY\min_i\ell_i$. Then there is $\lambda\in\Delta_N$ such that $\bar x=\sum_i\lambda_iu_i$ satisfies $P(\bar x)\le a$. Moreover, $\lambda$ can be computed from the planes alone, without further oracle calls.
\end{lemma}

\begin{proof}
For fixed $y$, the minimum of the linear function $\lambda\mapsto\sum_i\lambda_i\ell_i(y)$ over the simplex $\Delta_N$ is attained at a vertex, so $\min_i\ell_i(y)=\min_{\lambda\in\Delta_N}\sum_i\lambda_i\ell_i(y)$. The function $(\lambda,y)\mapsto\sum_i\lambda_i\ell_i(y)$ is continuous and affine in each argument, and $\Delta_N$ and $\cY$ are compact and convex. By the minimax theorem \citep{sion1958general},
\[
a=\max_{y\in\cY}\min_{\lambda\in\Delta_N}\sum_i\lambda_i\ell_i(y)=\min_{\lambda\in\Delta_N}\max_{y\in\cY}\sum_i\lambda_i\ell_i(y).
\]
Let $\lambda$ attain the minimum on the right; it depends only on the planes. For $\bar x=\sum_i\lambda_iu_i\in\cX$ and every $y\in\cY$, convexity of $f(\cdot,y)$ and $\ell_i\ge f(u_i,\cdot)$ give
\[
f(\bar x,y)\ \le\ \sum_i\lambda_if(u_i,y)\ \le\ \sum_i\lambda_i\ell_i(y)\ \le\ a .
\]
Taking the maximum over $y\in\cY$ gives $P(\bar x)\le a$.
\end{proof}

\section{Proof of \texorpdfstring{\cref{thm:general} and \cref{cor:onedim}}{the certificate-transport bounds}}\label{app:minmax-general}

Throughout this appendix \cref{ass:minmax} holds and $0<\eps<2B$. Volumes of subsets of $\cY$ are taken in $\aff\cY$, which has dimension $m$. As usual, bounds of the form $O(\log\frac B\eps)$ absorb additive constants, that is, they should be read as $O(1+\log\frac B\eps)$; note that $1+\log\frac B\eps>1-\log2>0$ because $\eps<2B$.

\paragraph{Notation for one run.}
We number the queries of the main loop of \cref{alg:ct} by $t=1,2,\ldots,N$. In query $t$ the algorithm
\begin{enumerate}[label=\arabic*.,leftmargin=1.6em,itemsep=1pt]
\item takes the center of gravity $y_t$ of the localizer $S_t=\{y\in\cY:\ U_t(y)\ge h\}$, where $U_t$ is the upper model before query $t$ and $h$ is the current threshold;
\item runs $(u_t,z_t,A)=\textsc{Block}(f(\cdot,y_t),q_b,K)$, which costs $K$ calls to $\Ox$;
\item calls $(v_t,g_t)=\Oy(u_t,y_t)$ and adds the plane $\ell_t(y)=v_t+\ip{g_t}{y-y_t}$ to the upper model;
\item if $v_t>h$, replaces the lower model by the transported one from \cref{cor:transport}, and otherwise leaves it unchanged.
\end{enumerate}
Here $A=K(K+3)/(8\kappa)$ is the same in every query, and we write $\vartheta=A/(1+A)$. Following \eqref{eq:balance}, we call query $t$ a \emph{cut} if $v_t\le h$ and a \emph{transport} if $v_t>h$. The quantities $a=\max_\cY U$ and $b$ always refer to their current values.

The proof has four parts. \Cref{lem:invariants} shows that the output is always certified by the width $a-b$. \Cref{lem:epochs,lem:thresholds} control the epochs and the position of the threshold. \Cref{lem:transports} bounds the number of transports, and \cref{lem:volume,lem:cuts} bound the number of cuts. The count of oracle calls then follows.

\subsection{Invariants and correctness}

\begin{lemma}[invariants]\label{lem:invariants}
At every moment of the run the following hold.
\begin{enumerate}[label=\textup{(\roman*)},leftmargin=2em,itemsep=1pt]
\item Every plane satisfies $\ell_i\ge f(u_i,\cdot)\ge\phi$ on $\cY$, and $U=\min_i\ell_i$ is concave and $\My$-Lipschitz on $\cY$. In particular, $a\ge\max_\cY\phi$.
\item $(b,x_b,y_b)$ is a lower model: $x_b\in\cX$, $y_b\in\cY$, and $f(x,y_b)\ge b+\frac\sigma2\nrm{x-x_b}^2$ for all $x\in\cX$. In particular, $\phi(y_b)\ge b$.
\item For the point $\bar x$ of \cref{lem:recovery}, $\Gap(\bar x,y_b)\le a-b$.
\item $a$ never increases.
\end{enumerate}
In particular, when the algorithm stops, $a-b\le\eps$ and the output satisfies $\Gap(\bar x,y_b)\le\eps$.
\end{lemma}

\begin{proof}
(i) Each plane comes from a call $\Oy(u_i,y_i)$, so $\ell_i\ge f(u_i,\cdot)$ on $\cY$ by concavity of $f(u_i,\cdot)$, and $f(u_i,\cdot)\ge\min_xf(x,\cdot)=\phi$. Each $\ell_i$ is affine with slope $g_i$, $\nrm{g_i}\le\My$, hence $\My$-Lipschitz; a minimum of concave $\My$-Lipschitz functions is concave and $\My$-Lipschitz. Finally, $U\ge\phi$ gives $a=\max_\cY U\ge\max_\cY\phi$.

(ii) The initial triple is a lower model by \cref{cor:start}. The lower model changes only in transports, and then the new triple is a lower model by \cref{cor:transport}. The bound on $\phi(y_b)$ follows by minimizing both sides over $x\in\cX$: $\phi(y_b)=\min_xf(x,y_b)\ge\min_x\big(b+\frac\sigma2\nrm{x-x_b}^2\big)=b$.

(iii) By \cref{lem:recovery}, applied to all planes collected so far (including $\ell_0$ from \cref{cor:start}), $P(\bar x)\le a$. Together with (ii), $\Gap(\bar x,y_b)=P(\bar x)-\phi(y_b)\le a-b$.

(iv) Adding a plane can only decrease $U=\min_i\ell_i$, and therefore $a=\max_\cY U$.

The last claim follows from (iii) and the stopping rule $a-b\le\eps$.
\end{proof}

Note that $b$ is not monotone for an arbitrary transport, since $b^+-b=\vartheta(v_t-b)$ is negative when $v_t<b$. It becomes monotone because transports happen only above the threshold, which lies above $b$ (\cref{lem:thresholds}).

\subsection{Epochs and the threshold}

The run is divided into epochs. Epoch $j=0,1,\ldots$ starts at a moment when the algorithm sets $W\gets a-b$ and $h\gets b+W/2$; let $a_j$ and $b_j$ be the values of $a$ and $b$ at that moment, and let
\[
W_j=a_j-b_j,\qquad h_j=b_j+\frac{W_j}2 .
\]
A new epoch starts before a query as soon as $a-b\le\frac23W_j$, and the algorithm stops as soon as $a-b\le\eps$. Consequently, immediately before every query of epoch $j$,
\begin{equation}\label{eq:inepoch}
a-b\ >\ \max\Big\{\eps,\ \frac23W_j\Big\}.
\end{equation}
For the first query of the epoch this holds because $a-b=W_j>\eps$.

\begin{lemma}[position of the threshold]\label{lem:thresholds}
Immediately before every query of epoch $j$,
\begin{equation}\label{eq:thresholds}
h_j-b\ >\ \frac{W_j}6
\qquad\text{and}\qquad
a-h_j\ >\ \frac{W_j}6 .
\end{equation}
Moreover, $b$ never decreases, and $a\le a_j$, $b\ge b_j$ during epoch $j$.
\end{lemma}

\begin{proof}
We first prove \eqref{eq:thresholds} under the assumption that $a\le a_j$ and $b\ge b_j$ at the moment in question, and then show that these inequalities indeed hold.

By \eqref{eq:inepoch} and $a\le a_j$,
\[
b\ <\ a-\frac23W_j\ \le\ a_j-\frac23W_j=b_j+\frac13W_j ,
\]
hence
\[
h_j-b\ >\ \Big(b_j+\frac{W_j}2\Big)-\Big(b_j+\frac{W_j}3\Big)=\frac{W_j}6 .
\]
Since $b\ge b_j$, we also have $h_j-b\le h_j-b_j=\frac{W_j}2$. Together with \eqref{eq:inepoch},
\[
a-h_j=(a-b)-(h_j-b)\ >\ \frac23W_j-\frac12W_j=\frac{W_j}6 .
\]

It remains to check $a\le a_j$ and $b\ge b_j$. The first holds because $a$ never increases (\cref{lem:invariants}(iv)). For the second, we show by induction over the queries of the epoch that $b$ does not decrease. At the start of the epoch $b=b_j$. If $b\ge b_j$ before a query, then \eqref{eq:thresholds} holds before it, so $h_j>b$. If the query is a transport, then $v_t>h_j>b$ and $b^+-b=\vartheta(v_t-b)>0$. If it is a cut, $b$ does not change. Hence $b\ge b_j$ before the next query as well. The value of $b$ does not change between epochs, so $b$ never decreases over the whole run.
\end{proof}

\begin{lemma}[number of epochs]\label{lem:epochs}
$W_{j+1}\le\frac23W_j$, and the number of epochs satisfies
\[
N_{\rm ep}\ \le\ 1+\Big\lceil\frac{\log(2B/\eps)}{\log(3/2)}\Big\rceil=O\Big(\log\frac B\eps\Big).
\]
\end{lemma}

\begin{proof}
Epoch $j+1$ starts when $a-b\le\frac23W_j$, and at that moment $W_{j+1}=a-b$. Hence $W_j\le(\frac23)^jW_0\le(\frac23)^j\cdot2B$, where $W_0\le2B$ by \cref{cor:start}. Every epoch starts with $W_j=a-b>\eps$, since otherwise the algorithm would have stopped. Therefore the last epoch $N_{\rm ep}-1$ satisfies $\eps<(\frac23)^{N_{\rm ep}-1}\cdot2B$, that is, $N_{\rm ep}-1<\log(2B/\eps)/\log(3/2)$.
\end{proof}

\subsection{Transports}

\begin{lemma}[number of transports]\label{lem:transports}
Every epoch contains at most $\lceil2/\vartheta\rceil$ transports, and
\[
\frac2\vartheta=2\Big(1+\frac1A\Big)\ \le\ 2+\frac{16\kappa}{K^2}.
\]
\end{lemma}

\begin{proof}
\emph{Each transport increases $b$ by a fixed fraction of $W_j$.} In a transport of epoch $j$, $v_t>h_j$, so by \cref{cor:transport} and \eqref{eq:thresholds}
\[
b^+-b=\vartheta(v_t-b)\ >\ \vartheta(h_j-b)\ >\ \frac{\vartheta W_j}6 .
\]
\emph{The total increase of $b$ within an epoch is limited.} Before every query of epoch $j$ we showed in the proof of \cref{lem:thresholds} that $b<b_j+\frac13W_j$. Suppose epoch $j$ contains $n$ transports. Before the last of them, $b$ has already grown by more than $(n-1)\frac{\vartheta W_j}6$ and is still below $b_j+\frac13W_j$. Hence $(n-1)\frac{\vartheta}6<\frac13$, that is, $n<1+\frac2\vartheta$, and so $n\le\lceil2/\vartheta\rceil$.

\emph{The bound on $2/\vartheta$.} By definition, $\frac1\vartheta=\frac{1+A}A=1+\frac1A$, and $A=\frac{K(K+3)}{8\kappa}\ge\frac{K^2}{8\kappa}$, so $\frac1A\le\frac{8\kappa}{K^2}$.
\end{proof}

\subsection{Cuts}

The number of cuts is controlled by the volume of the localizer. Let
\[
V(t)=\vol S_t,\qquad S_t=\{y\in\cY:\ U_t(y)\ge h_{j(t)}\},
\]
where $j(t)$ is the epoch of query $t$, and $U_t$ is the upper model immediately before query $t$.

\begin{lemma}[volume of the localizer]\label{lem:volume}
The following hold.
\begin{enumerate}[label=\textup{(\alph*)},leftmargin=2em,itemsep=1pt]
\item \emph{Lower bound.} $V(t)\ge\big(\frac{\eps}{6B}\big)^m\vol\cY$ for every $t$. In particular, $S_t$ is a convex body in $\aff\cY$, and its center of gravity $y_t$ is well defined and lies in $\cY$.
\item \emph{Cuts shrink the localizer.} If query $t$ is a cut and query $t+1$ belongs to the same epoch, then $V(t+1)\le(1-e^{-1})V(t)$.
\item \emph{Transports do not enlarge it.} If query $t$ is a transport and query $t+1$ belongs to the same epoch, then $V(t+1)\le V(t)$.
\item \emph{A change of epoch enlarges it by at most $6^m$.} If query $t$ is the last query of epoch $j$ and query $t+1$ is the first query of epoch $j+1$, then $V(t+1)\le6^mV(t)$.
\end{enumerate}
\end{lemma}

\begin{proof}
Write $j=j(t)$, $U=U_t$, and let $a$ be the value of $\max_\cY U_t$.

(a) Let $y^\circ\in\argmax_\cY U$ and $\lambda=\min\{1,(a-h_j)/B\}$. Take any $y''\in\cY$ and put $y'=y^\circ+\lambda(y''-y^\circ)$, which lies in $\cY$ by convexity. By \cref{lem:invariants}(i), $U$ is $\My$-Lipschitz, so $U(y'')\ge a-\My\nrm{y''-y^\circ}\ge a-B$. By concavity of $U$,
\[
U(y')\ \ge\ (1-\lambda)U(y^\circ)+\lambda U(y'')\ \ge\ (1-\lambda)a+\lambda(a-B)=a-\lambda B\ \ge\ h_j ,
\]
where the last step uses $\lambda\le(a-h_j)/B$. Hence $y'\in S_t$, that is, $S_t$ contains the homothetic copy $y^\circ+\lambda(\cY-y^\circ)$ of $\cY$, and $V(t)\ge\lambda^m\vol\cY$. By \eqref{eq:thresholds}, $a-h_j>W_j/6$, and $W_j>\eps$ by \eqref{eq:inepoch}; hence $a-h_j>\eps/6$. Since also $\eps/(6B)<1$, we get $\lambda\ge\eps/(6B)$. Finally, $S_t$ is convex (a superlevel set of a concave function intersected with $\cY$), closed, and has positive volume, so it is a convex body.

(b) The threshold is the same for queries $t$ and $t+1$, and $U_{t+1}=\min\{U_t,\ell_t\}$. Hence
\[
S_{t+1}=S_t\cap\{y\in\cY:\ \ell_t(y)\ge h_j\}.
\]
If $\ell_t(y)\ge h_j$, then $\ip{g_t}{y-y_t}=\ell_t(y)-v_t\ge h_j-v_t\ge0$, because query $t$ is a cut. So $S_{t+1}\subseteq S_t\cap H$ with $H=\{y:\ip{g_t}{y-y_t}\ge0\}$. The boundary of $H$ passes through the center of gravity $y_t$ of $S_t$, and \cref{lem:grunbaum} applied in $\aff\cY$ gives $V(t+1)\le(1-e^{-1})V(t)$.

This argument needs $H\cap\aff\cY$ to be a proper half-space of $\aff\cY$, that is, $g_t$ must not be orthogonal to $\aff\cY$. If it is, then $\ell_t$ is constant on $\cY$ and equal to $v_t\le h_j$, so after query $t$ we have $a\le v_t\le h_j$ and $a-b\le h_j-b\le h_j-b_j=\frac{W_j}2<\frac23W_j$. Then query $t$ is the last query of its epoch, and (b) does not apply to it.

(c) Again the threshold is unchanged and $U_{t+1}\le U_t$, so $S_{t+1}\subseteq S_t$.

(d) Let $a'$ be the value of $a$ before query $t$ and $y^\circ\in\argmax_\cY U_t$. Since $U_{t+1}\le U_t$,
\[
S_{t+1}=\{U_{t+1}\ge h_{j+1}\}\subseteq\{y\in\cY:\ U_t(y)\ge h_{j+1}\}.
\]
If $h_{j+1}\ge h_j$, the right-hand side is contained in $S_t$ and we are done. Otherwise $h_{j+1}<h_j<a'$, where the last inequality is \eqref{eq:thresholds}. Put
\[
\lambda=\frac{a'-h_j}{a'-h_{j+1}}\in(0,1].
\]
Take $y\in\cY$ with $U_t(y)\ge h_{j+1}$. By concavity,
\[
U_t\big(y^\circ+\lambda(y-y^\circ)\big)\ \ge\ (1-\lambda)a'+\lambda h_{j+1}=a'-\lambda(a'-h_{j+1})=h_j ,
\]
so $y^\circ+\lambda(y-y^\circ)\in S_t$, that is, $y\in y^\circ+\lambda^{-1}(S_t-y^\circ)$. Hence $\{U_t\ge h_{j+1}\}$ is contained in a homothetic copy of $S_t$ with ratio $\lambda^{-1}$, and $V(t+1)\le\lambda^{-m}V(t)$. It remains to bound $\lambda^{-1}$. For the numerator, $a'\le a_j$ and $h_{j+1}\ge b_{j+1}\ge b_j$ (\cref{lem:thresholds}), so $a'-h_{j+1}\le a_j-b_j=W_j$. For the denominator, $a'-h_j>W_j/6$ by \eqref{eq:thresholds}. Therefore $\lambda^{-1}<6$.
\end{proof}

\begin{lemma}[number of cuts]\label{lem:cuts}
The total number $C$ of cuts satisfies
\[
C\ \le\ N_{\rm ep}+\frac{m\log\frac{6B}{\eps}+mN_{\rm ep}\log6}{\log\frac{e}{e-1}}=O\Big(m\log\frac B\eps\Big).
\]
\end{lemma}

\begin{proof}
Follow $V(t)$ from $t=1$ to $t=N$. At the start, $V(1)\le\vol\cY$. For each $t<N$, \cref{lem:volume}(b)--(d) gives:
\begin{itemize}[leftmargin=1.2em,itemsep=1pt]
\item a factor at most $1-e^{-1}$ if query $t$ is a cut and not the last query of its epoch;
\item a factor at most $1$ if query $t$ is a transport and not the last query of its epoch;
\item a factor at most $6^m$ if query $t$ is the last query of its epoch.
\end{itemize}
Every epoch has exactly one last query, and query $N$ is the last query of the last epoch. Hence at most $N_{\rm ep}$ cuts are last queries of their epochs, at least $C-N_{\rm ep}$ cuts contribute the factor $1-e^{-1}$, and at most $N_{\rm ep}-1\le N_{\rm ep}$ changes of epoch contribute the factor $6^m$. Combining this with the lower bound of \cref{lem:volume}(a) at $t=N$,
\[
\Big(\frac{\eps}{6B}\Big)^m\vol\cY\ \le\ V(N)\ \le\ (1-e^{-1})^{C-N_{\rm ep}}\,6^{mN_{\rm ep}}\vol\cY .
\]
Dividing by $\vol\cY$ and taking logarithms,
\[
(C-N_{\rm ep})\log\frac{e}{e-1}\ \le\ m\log\frac{6B}{\eps}+mN_{\rm ep}\log6 ,
\]
which is the claim. The $O(\cdot)$ bound follows from $N_{\rm ep}=O(\log\frac B\eps)$ (\cref{lem:epochs}).
\end{proof}

\subsection{Proof of \texorpdfstring{\cref{thm:general}}{the theorem}}

\emph{Correctness.} By \cref{lem:invariants}, the output satisfies $\Gap(\bar x,y_b)\le\eps$. This holds for every $K\ge1$: the block length enters only the number of queries.

\emph{Number of queries.} Every query is a cut or a transport. By \cref{lem:cuts,lem:transports,lem:epochs},
\[
N\ \le\ C+N_{\rm ep}\Big\lceil\frac2\vartheta\Big\rceil
\ \le\ O\Big(m\log\frac B\eps\Big)+N_{\rm ep}\Big(3+\frac{16\kappa}{K^2}\Big)
\ =\ O\Big(\Big(m+\frac{\kappa}{K^2}\Big)\log\frac B\eps\Big),
\]
where the constant term $3N_{\rm ep}$ is absorbed into $mN_{\rm ep}$ because $m\ge1$. In particular, the algorithm terminates.

\emph{Oracle calls.} Initialization costs $I_0$ calls to $\Ox$ and one call to $\Oy$ (\cref{cor:start}), and every query costs $K$ calls to $\Ox$ and one call to $\Oy$. Therefore
\[
\Ty=1+N=O\Big(\Big(m+\frac{\kappa}{K^2}\Big)\log\frac B\eps\Big),
\qquad
\Tx=I_0+KN=I_0+O\Big(\Big(mK+\frac{\kappa}{K}\Big)\log\frac B\eps\Big),
\]
where we used $K\cdot\frac{\kappa}{K^2}=\frac\kappa K$. This is the first part of \cref{thm:general}.

\emph{The default block length.} Let $K=\lceil\sqrt{8\kappa/m}\,\rceil$. Then $K\ge\sqrt{8\kappa/m}$, so $\frac{\kappa}{K^2}\le\frac m8$ and $\frac\kappa K\le\sqrt{\frac{m\kappa}8}$. Also $K\le\sqrt{8\kappa/m}+1$, so $mK\le\sqrt{8m\kappa}+m$. Substituting,
\[
\Ty=O\Big(m\log\frac B\eps\Big),
\qquad
\Tx=I_0+O\Big(\big(m+\sqrt{m\kappa}\big)\log\frac B\eps\Big),
\]
which is the second part of \cref{thm:general}. \hfill$\square$

\subsection{Proof of \texorpdfstring{\cref{cor:onedim}}{the one-dimensional corollary}}

For $m=1$ the default block length is $K=\lceil\sqrt{8\kappa}\,\rceil$. Since $\kappa\ge1$, we have $m+\sqrt{m\kappa}=1+\sqrt\kappa\le2\sqrt\kappa$, and the second part of \cref{thm:general} gives
\[
\Ty=O\Big(\log\frac B\eps\Big),
\qquad
\Tx=I_0+O\Big(\sqrt\kappa\log\frac B\eps\Big)=O\Big(\sqrt\kappa\Big[1+\logp\frac{\mu\Rx^2}{B}+\log\frac B\eps\Big]\Big),
\]
where we inserted $I_0$ from \cref{cor:start}. \hfill$\square$

\begin{remark}[the block length and the limit of this analysis]
No correctness argument above uses the value of $K$: the transported triple is a lower model for every $K$ (\cref{cor:transport}), and the certificate $\Gap\le a-b$ can be checked at any moment (\cref{lem:invariants}). The block length affects only the number of queries, through the trade-off between the $mK$ gradients spent on cuts and the $\kappa/K$ gradients spent on transports. The choice $K\asymp\sqrt{\kappa/m}$ balances these two terms, so to go below $\sqrt{m\kappa}$ in general one needs a different measure of progress. Under strong concavity in $y$, the method of \cref{sec:strong} provides one.
\end{remark}

\section{Proof of \texorpdfstring{\cref{thm:cross}}{the strongly concave bound}: accelerated two-point method}\label{app:minmax-cross}

Throughout this appendix \cref{ass:minmax,ass:cross} hold. Recall that $f(x,\cdot)$ is $\nu$-strongly concave, $\nabla_xf$ is $\ell$-Lipschitz in $y$, and $\chi=\ell^2/(\mu\nu)$. We use the constants
\[
\gamma_\chi\defeq\frac{\ell^2}{\nu}=\mu\chi,\qquad \bar\chi\defeq1+\chi,\qquad \Gamma\defeq\mu+\gamma_\chi=\mu \bar\chi .
\]
Since $\cX$ and $\cY$ are compact and $f$ is strongly convex in $x$ and strongly concave in $y$, the saddle point $(x^\star,y^\star)$ exists and is unique, and
\[
P^\star\defeq P(x^\star)=f(x^\star,y^\star)=\phi(y^\star).
\]
For $x\in\cX$, let $y^\star(x)=\argmax_\cY f(x,\cdot)$; the maximizer is unique by strong concavity.

We will use two standard consequences of strong convexity and concavity on a convex set. If $x'$ minimizes a $\mu$-strongly convex function $h$ over $\cX$, then first-order optimality gives $\ip{\nabla h(x')}{x-x'}\ge0$ for $x\in\cX$, and therefore
\begin{equation}\label{eq:qg}
h(x)-h(x')\ \ge\ \frac\mu2\nrm{x-x'}^2\qquad\forall x\in\cX .
\end{equation}
Symmetrically, if $y'$ maximizes a $\nu$-strongly concave function $h$ over $\cY$, then $h(y')-h(y)\ge\frac\nu2\nrm{y-y'}^2$ for all $y\in\cY$. In particular, $P=\max_yf(\cdot,y)$ is $\mu$-strongly convex as a maximum of $\mu$-strongly convex functions, and $P(x)-P^\star\ge\frac\mu2\nrm{x-x^\star}^2$.

The method analyzed in this appendix is \cref{alg:cross}. Its outer loop halves an upper bound $E$ on the primal error $P(w)-P^\star$; each halving is one \emph{outer block} of $K$ steps, and each step makes one call to a solver $\mathsf{MaxY}$ of the problem in $y$ and one run of \cref{alg:block} on a proximal subproblem.

\begin{algorithm}[tb]
\caption{Two-point acceleration under strong concavity}\label{alg:cross}
\begin{algorithmic}[1]
\Require $\eps>0$; constants $L,\mu,\chi,B$; a solver $\mathsf{MaxY}(x,\delta)$ for the problem in $y$
\Ensure $(\hat x,\hat y)$ with $\Gap(\hat x,\hat y)\le\eps$
\State $\bar\chi\gets1+\chi$;\ \ $\Gamma\gets\mu \bar\chi$;\ \ $L_\Phi\gets1+L/(4\Gamma)$
\State $k\gets\lceil2\sqrt{L_\Phi}\,\rceil$;\ \ $A_\Phi\gets k(k+3)/(4L_\Phi)$;\ \ $\theta\gets A_\Phi/(1+A_\Phi)$
\State $K\gets\lceil8\sqrt{\bar\chi/\theta}\,\rceil$;\ \ $c_\delta\gets \bar\chi/(4K^3)$;\ \ $E\gets2B$
\State Find $w$ with $P(w)-P^\star\le E$ (\cref{lem:init})
\While{$E>\eps/(6\bar\chi^2)$}
  \State $A\gets0$;\ \ $u,z\gets w$;\ \ $\delta\gets c_\delta E$
  \For{$i=1,\ldots,K$}
    \State Find $a>0$ from $a^2=\theta(A+a)/(4\Gamma)$;\ \ $A^+\gets A+a$
    \State $w_i(p)\gets(Au+ap)/A^+$;\ \ $v\gets w_i(z)$
    \State $y_i\gets\mathsf{MaxY}(v,\delta)$ \Comment{$P(v)-f(v,y_i)\le\delta$}
    \State $\Phi_i(p)\gets(A^+/\theta)\,f(w_i(p),y_i)+\frac12\nrm{p-z}^2$
    \State $(p_i,z^+,\_)\gets\textsc{Block}\big(\Phi_i,\ \tfrac12\nrm{\cdot-z}^2,\ k\big)$ \Comment{$\mu_F=1+\frac{\mu}{4\Gamma}$, $L_F=L_\Phi$}
    \State $u\gets w_i(p_i)$;\ \ $z\gets z^+$;\ \ $A\gets A^+$
  \EndFor
  \State $w\gets u$;\ \ $E\gets E/2$
\EndWhile
\State $\hat x\gets w$;\ \ $\hat y\gets\mathsf{MaxY}(\hat x,c_\delta E)$;\ \ \Return $(\hat x,\hat y)$
\end{algorithmic}
\end{algorithm}

\subsection{Cross-Lipschitz estimate}

The coupling assumption enters the proof only through the following inequality. For $x,v\in\cX$ put $\psi_{x,v}(y)=f(x,y)-f(v,y)$. Then
\begin{equation}\label{eq:crosslip}
\big|\psi_{x,v}(y)-\psi_{x,v}(y')\big|\ \le\ \ell\,\nrm{x-v}\,\nrm{y-y'}\qquad\forall y,y'\in\cY .
\end{equation}
Indeed, let $v_t=v+t(x-v)\in\cX$ for $t\in[0,1]$. By the fundamental theorem of calculus applied to $t\mapsto f(v_t,y)$,
\[
\psi_{x,v}(y)=f(x,y)-f(v,y)=\int_0^1\ip{\nabla_xf(v_t,y)}{x-v}\,dt ,
\]
and similarly for $y'$. Subtracting and using the Cauchy--Schwarz inequality and \cref{ass:cross},
\[
\big|\psi_{x,v}(y)-\psi_{x,v}(y')\big|=\Big|\int_0^1\ip{\nabla_xf(v_t,y)-\nabla_xf(v_t,y')}{x-v}\,dt\Big|\le\ell\nrm{y-y'}\nrm{x-v}.
\]

\subsection{A two-sided model of the primal function}

\begin{lemma}[two-sided model]\label{lem:model}
Let $v\in\cX$, and let $\tilde y\in\cY$ satisfy $P(v)-f(v,\tilde y)\le\delta$. Then for all $x\in\cX$
\[
f(x,\tilde y)\ \le\ P(x)\ \le\ f(x,\tilde y)+\gamma_\chi\nrm{x-v}^2+2\delta .
\]
\end{lemma}

\begin{proof}
The left inequality is the definition of $P$. For the right one, fix $x\in\cX$ and write $d=\nrm{x-v}$, $y_x=y^\star(x)$, $y_v=y^\star(v)$, and $\psi=\psi_{x,v}$.

\emph{Step 1: two consequences of strong concavity at $v$.} Since $y_v$ maximizes the $\nu$-strongly concave function $f(v,\cdot)$,
\[
\frac\nu2\nrm{\tilde y-y_v}^2\le P(v)-f(v,\tilde y)\le\delta
\quad\Longrightarrow\quad
\nrm{\tilde y-y_v}^2\le\frac{2\delta}\nu,
\]
and, with $y=y_x$,
\[
f(v,y_x)\ \le\ P(v)-\frac\nu2\nrm{y_x-y_v}^2 .
\]

\emph{Step 2: splitting the error.} Since $f(x,\cdot)=f(v,\cdot)+\psi$,
\[
P(x)-f(x,\tilde y)=f(x,y_x)-f(x,\tilde y)=\big[\psi(y_x)-\psi(\tilde y)\big]+\big[f(v,y_x)-f(v,\tilde y)\big].
\]
By \eqref{eq:crosslip} and the triangle inequality, the first bracket is at most $\ell d\,\nrm{y_x-\tilde y}\le\ell d\big(\nrm{y_x-y_v}+\nrm{y_v-\tilde y}\big)$. By Step~1 and $P(v)-f(v,\tilde y)\le\delta$, the second bracket is at most $\delta-\frac\nu2\nrm{y_x-y_v}^2$. Hence, with $s=\nrm{y_x-y_v}$,
\[
P(x)-f(x,\tilde y)\ \le\ \Big[\ell ds-\frac\nu2s^2\Big]+\ell d\,\nrm{y_v-\tilde y}+\delta .
\]

\emph{Step 3: bounding the two terms.} The quadratic in $s$ is maximized at $s=\ell d/\nu$, so $\ell ds-\frac\nu2s^2\le\frac{\ell^2d^2}{2\nu}$. By the inequality $\alpha\beta\le\frac{\alpha^2}{2\nu}+\frac{\nu\beta^2}2$ and Step~1,
\[
\ell d\,\nrm{y_v-\tilde y}\ \le\ \frac{\ell^2d^2}{2\nu}+\frac\nu2\nrm{y_v-\tilde y}^2\ \le\ \frac{\ell^2d^2}{2\nu}+\delta .
\]
Adding up, $P(x)-f(x,\tilde y)\le\frac{\ell^2}\nu d^2+2\delta=\gamma_\chi\nrm{x-v}^2+2\delta$.
\end{proof}

The lemma says that one approximate maximizer $\tilde y$ computed at $v$ gives a model of $P$ that is exact from below and $\gamma_\chi$-smooth from above up to $2\delta$. This is an inexact model of $P$ in the sense of \citet{devolder2014first,stonyakin2021inexact}.

\subsection{From primal accuracy to the full gap}

\begin{lemma}[from primal accuracy to the gap]\label{lem:gapcross}
If $x\in\cX$ and $y\in\cY$ satisfy $P(x)-P^\star\le E$ and $P(x)-f(x,y)\le\delta$, then
\[
\Gap(x,y)\ \le\ E+\bar\chi (2\delta+4\chi E).
\]
\end{lemma}

\begin{proof}
Split the gap as $\Gap(x,y)=\big(P(x)-P^\star\big)+\big(P^\star-\phi(y)\big)$. The first term is at most $E$, so we bound the second one.

\emph{Step 1: how far $y$ is from $y^\star$.} Let $D=f(x^\star,y^\star)-f(x^\star,y)$. Since $y^\star$ maximizes the $\nu$-strongly concave function $f(x^\star,\cdot)$, we have $D\ge\frac\nu2\nrm{y-y^\star}^2$. To bound $D$ from above, add and subtract $f(x,y^\star)$ and $f(x,y)$:
\[
D=\underbrace{\big[f(x,y^\star)-f(x,y)\big]}_{\le P(x)-f(x,y)\le\delta}
+\underbrace{\big[\psi_{x,x^\star}(y)-\psi_{x,x^\star}(y^\star)\big]}_{\le\ell\nrm{y-y^\star}\nrm{x-x^\star}\text{ by }\eqref{eq:crosslip}} .
\]
Now $\nrm{y-y^\star}^2\le2D/\nu$ and $\nrm{x-x^\star}^2\le2E/\mu$ (strong convexity of $P$). Hence
\[
\ell\nrm{y-y^\star}\nrm{x-x^\star}\ \le\ \ell\sqrt{\frac{2D}\nu\cdot\frac{2E}\mu}=2\sqrt{\chi DE}\ \le\ \frac D2+2\chi E ,
\]
where the last step is $2\sqrt{\alpha\beta}\le\alpha+\beta$ with $\alpha=D/2$ and $\beta=2\chi E$. Thus $D\le\delta+\frac D2+2\chi E$, that is,
\[
D\ \le\ 2\delta+4\chi E .
\]

\emph{Step 2: from $D$ to $P^\star-\phi(y)$.} Let $x_y=\argmin_\cX f(\cdot,y)$, so that $\phi(y)=f(x_y,y)$. Then
\[
P^\star-\phi(y)=f(x^\star,y^\star)-f(x_y,y)=D+\big[f(x^\star,y)-f(x_y,y)\big].
\]
By $\mu$-strong convexity of $f(\cdot,y)$,
\[
f(x^\star,y)-f(x_y,y)\ \le\ \ip{\nabla_xf(x^\star,y)}{x^\star-x_y}-\frac\mu2\nrm{x^\star-x_y}^2 .
\]
Since $x^\star$ minimizes $f(\cdot,y^\star)$ over $\cX$, first-order optimality gives $\ip{\nabla_xf(x^\star,y^\star)}{x^\star-x_y}\le0$, and we may subtract this term:
\[
f(x^\star,y)-f(x_y,y)\ \le\ \ip{\nabla_xf(x^\star,y)-\nabla_xf(x^\star,y^\star)}{x^\star-x_y}-\frac\mu2\nrm{x^\star-x_y}^2
\ \le\ \ell\nrm{y-y^\star}s-\frac\mu2s^2 ,
\]
with $s=\nrm{x^\star-x_y}$. Maximizing over $s$ and using Step~1,
\[
f(x^\star,y)-f(x_y,y)\ \le\ \frac{\ell^2\nrm{y-y^\star}^2}{2\mu}\ \le\ \frac{\ell^2}{2\mu}\cdot\frac{2D}\nu=\chi D .
\]
Therefore $P^\star-\phi(y)\le(1+\chi)D=\bar\chi D\le \bar\chi (2\delta+4\chi E)$.
\end{proof}

\subsection{The inner two-point inequality}

Each step of the outer method minimizes a proximal subproblem only approximately, by one run of \cref{alg:block}. The next lemma turns the guarantee of \cref{lem:block} into the form needed by the outer analysis. It returns two points: $p$, at which the objective is evaluated, and $z$, around which the quadratic lower bound is centered.

\begin{lemma}[two-point inequality]\label{lem:twopoint}
Let $g$ be convex on $\cX$, let $z_-\in\cX$ and $\theta_0>0$, and suppose that
\[
\Phi(z)=\frac1{\theta_0}g(z)+\frac12\nrm{z-z_-}^2
\]
is $\sigma_\Phi$-strongly convex and $L_\Phi$-smooth on $\cX$. Run the method of \cref{lem:block} on $F=\Phi$ with prior $\pi=\frac12\nrm{\cdot-z_-}^2$ (that is, $c=0$, $\gamma=1$, $\bar x=z_-$) for $k=\lceil2\sqrt{L_\Phi}\,\rceil$ steps, and let $p=u_k$ and $z=z_k$ be its output. Then
\[
A_\Phi\defeq A_k=\frac{k(k+3)}{4L_\Phi}\ \ge\ 1,\qquad \theta\defeq\frac{A_\Phi}{1+A_\Phi}\in\Big[\frac12,1\Big),
\]
and for all $w\in\cX$
\begin{equation}\label{eq:twopoint}
\frac{\theta}{\theta_0}g(p)+\frac\theta2\nrm{p-z_-}^2+\frac{1+\theta(\sigma_\Phi-1)}{2}\nrm{w-z}^2\ \le\ \frac\theta{\theta_0}g(w)+\frac12\nrm{w-z_-}^2 .
\end{equation}
\end{lemma}

\begin{proof}
The bound $A_\Phi\ge k^2/(4L_\Phi)\ge1$ follows from $k\ge2\sqrt{L_\Phi}$, and then $\theta=A_\Phi/(1+A_\Phi)\ge\frac12$. By \eqref{eq:blockineq} with $\mu_F=\sigma_\Phi$, for all $w\in\cX$,
\[
\frac12\nrm{w-z_-}^2+A_\Phi\,\Phi(w)\ \ge\ A_\Phi\,\Phi(p)+\frac{1+\sigma_\Phi A_\Phi}2\nrm{w-z}^2 .
\]
Substitute the definition of $\Phi$ on both sides. On the left, the two quadratic terms combine into $\frac{1+A_\Phi}2\nrm{w-z_-}^2$:
\[
\frac{1+A_\Phi}2\nrm{w-z_-}^2+\frac{A_\Phi}{\theta_0}g(w)\ \ge\ \frac{A_\Phi}{\theta_0}g(p)+\frac{A_\Phi}2\nrm{p-z_-}^2+\frac{1+\sigma_\Phi A_\Phi}2\nrm{w-z}^2 .
\]
Divide by $1+A_\Phi$ and use $\frac{A_\Phi}{1+A_\Phi}=\theta$ and
\[
\frac{1+\sigma_\Phi A_\Phi}{1+A_\Phi}=\frac1{1+A_\Phi}+\sigma_\Phi\frac{A_\Phi}{1+A_\Phi}=(1-\theta)+\sigma_\Phi\theta=1+\theta(\sigma_\Phi-1).
\]
This is \eqref{eq:twopoint}.
\end{proof}

We apply \cref{lem:twopoint} with $\theta_0=\theta$. This is not circular: $\theta$ is determined by $L_\Phi$ alone, and below $L_\Phi$ is a fixed constant that does not depend on $\theta_0$ (see \eqref{eq:phiconst}). So we first compute $L_\Phi$, then $k$, $A_\Phi$ and $\theta$, and only then define the subproblems with $\theta_0=\theta$. With $\theta_0=\theta$, \eqref{eq:twopoint} reads
\begin{equation}\label{eq:twopoint2}
g(p)+\frac\theta2\nrm{p-z_-}^2+\frac{1+\theta(\sigma_\Phi-1)}{2}\nrm{w-z}^2\ \le\ g(w)+\frac12\nrm{w-z_-}^2\qquad\forall w\in\cX .
\end{equation}

\subsection{One outer block}

An outer block takes a point $w\in\cX$ with $P(w)-P^\star\le E$ and returns a point $u_K$ with $P(u_K)-P^\star\le\frac{9}{32}E$, that is, it more than halves the primal error.

\paragraph{Parameters.} Set
\begin{equation}\label{eq:phiconst}
\begin{gathered}
L_\Phi=1+\frac{L}{4\Gamma},\qquad\sigma_\Phi=1+\frac\mu{4\Gamma},\qquad k=\lceil2\sqrt{L_\Phi}\,\rceil,\qquad A_\Phi=\frac{k(k+3)}{4L_\Phi},\qquad\theta=\frac{A_\Phi}{1+A_\Phi},\\
K=\Big\lceil8\sqrt{\bar\chi/\theta}\,\Big\rceil,\qquad c_\delta=\frac{\bar\chi}{4K^3},\qquad\delta=c_\delta E .
\end{gathered}
\end{equation}

\paragraph{Weights.} Let $A_0=0$ and, for $i=1,\ldots,K$, let $a_i>0$ be the positive root of
\[
a_i^2=\frac{\theta}{4\Gamma}\,(A_{i-1}+a_i),\qquad A_i=A_{i-1}+a_i ,
\]
so that $a_i^2/A_i=\theta/(4\Gamma)$.

\paragraph{Steps.} Start from $u_0=z_0=w$. For $z\in\cX$ put $w_i(z)=\frac{A_{i-1}u_{i-1}+a_iz}{A_i}\in\cX$, and let $v_i=w_i(z_{i-1})$. For $i=1,\ldots,K$:
\begin{enumerate}[label=\arabic*.,leftmargin=1.6em,itemsep=1pt]
\item Using $\Oy$, find $y_i\in\cY$ with $P(v_i)-f(v_i,y_i)\le\delta$, for example by \cref{lem:cog} applied to the concave function $f(v_i,\cdot)$.
\item Apply \cref{lem:twopoint} to $g_i(z)=A_if(w_i(z),y_i)$ with $\theta_0=\theta$ and $z_-=z_{i-1}$, that is, to
\[
\Phi_i(z)=\frac{A_i}\theta f(w_i(z),y_i)+\frac12\nrm{z-z_{i-1}}^2 .
\]
Let $p_i$ and $z_i$ be its output, and put $u_i=w_i(p_i)$.
\end{enumerate}
The function $g_i$ is convex as a composition of the convex function $f(\cdot,y_i)$ with an affine map. The gradient $\nabla\Phi_i(z)=\frac{a_i}{\theta}\nabla_xf(w_i(z),y_i)+z-z_{i-1}$ costs one call to $\Ox$.

\paragraph{The constants of $\Phi_i$.} The map $z\mapsto w_i(z)$ is affine with linear part $\frac{a_i}{A_i}I$. Hence $z\mapsto\frac{A_i}\theta f(w_i(z),y_i)$ is strongly convex and smooth with constants
\[
\frac{A_i}\theta\cdot\frac{a_i^2}{A_i^2}\cdot\mu=\frac{a_i^2}{\theta A_i}\,\mu=\frac{\mu}{4\Gamma}
\qquad\text{and}\qquad
\frac{a_i^2}{\theta A_i}\,L=\frac{L}{4\Gamma},
\]
by the choice of $a_i$. Adding the $1$-strongly convex and $1$-smooth term $\frac12\nrm{z-z_{i-1}}^2$ shows that $\Phi_i$ is $\sigma_\Phi$-strongly convex and $L_\Phi$-smooth, with the constants in \eqref{eq:phiconst}. These constants do not depend on $i$ or $\theta$.

\begin{lemma}[growth of the weights]\label{lem:weights}
For $i\ge1$,
\[
\frac{\sqrt\theta\,(i+1)}{4\sqrt\Gamma}\ \le\ \sqrt{A_i}\ \le\ \frac{\sqrt\theta\, i}{2\sqrt\Gamma}.
\]
Consequently, $\sum_{i=1}^KA_i\le\frac{\theta K^3}{4\Gamma}$ and $A_K\ge\frac{\theta(K+1)^2}{16\Gamma}\ge\frac4\mu$.
\end{lemma}

\begin{proof}
Write $\rho=\sqrt{\theta/(4\Gamma)}=\frac{\sqrt\theta}{2\sqrt\Gamma}$, so $a_i=\rho\sqrt{A_i}$. For $i=1$, $A_1=a_1=\rho\sqrt{A_1}$ gives $\sqrt{A_1}=\rho$, which matches both bounds. For $i\ge2$,
\[
\sqrt{A_i}-\sqrt{A_{i-1}}=\frac{A_i-A_{i-1}}{\sqrt{A_i}+\sqrt{A_{i-1}}}=\frac{\rho\sqrt{A_i}}{\sqrt{A_i}+\sqrt{A_{i-1}}}\in\Big[\frac\rho2,\ \rho\Big],
\]
because $0\le A_{i-1}\le A_i$. Summing from $2$ to $i$ gives $\rho\big(1+\frac{i-1}2\big)\le\sqrt{A_i}\le\rho i$, which is the claim since $\rho(1+\frac{i-1}2)=\frac{\rho(i+1)}2$.

For the sum, $A_i\le\rho^2i^2$ and $\sum_{i=1}^Ki^2=\frac{K(K+1)(2K+1)}6\le K^3$; the last inequality is equivalent to $(4K+1)(K-1)\ge0$. Hence $\sum_iA_i\le\rho^2K^3=\frac{\theta K^3}{4\Gamma}$. For the last claim, $A_K\ge\frac{\rho^2(K+1)^2}4=\frac{\theta(K+1)^2}{16\Gamma}$, and $K\ge8\sqrt{\bar\chi/\theta}$ gives $(K+1)^2\ge K^2\ge\frac{64\bar\chi}\theta=\frac{64\Gamma}{\theta\mu}$, so $A_K\ge\frac4\mu$.
\end{proof}

\begin{lemma}[potential]\label{lem:potential}
Let $V_i=A_i\big(P(u_i)-P^\star\big)+\frac12\nrm{z_i-x^\star}^2$. Then for $i=1,\ldots,K$
\begin{equation}\label{eq:potential}
V_i+\frac\theta4\nrm{p_i-z_{i-1}}^2+\frac{\theta\mu}{8\Gamma}\nrm{z_i-x^\star}^2\ \le\ V_{i-1}+2A_i\delta ,
\end{equation}
and consequently $P(u_K)-P^\star\le\frac{9}{32}E$.
\end{lemma}

\begin{proof}
Fix $i$. We combine three inequalities.

\emph{(1) The inner run.} Apply \eqref{eq:twopoint2} with $g=g_i$, $z_-=z_{i-1}$, $p=p_i$, $z=z_i$ and $w=x^\star$. Since $g_i(p_i)=A_if(u_i,y_i)$ and $\sigma_\Phi-1=\frac\mu{4\Gamma}$,
\[
A_if(u_i,y_i)+\frac\theta2\nrm{p_i-z_{i-1}}^2+\frac12\Big(1+\frac{\theta\mu}{4\Gamma}\Big)\nrm{x^\star-z_i}^2\ \le\ A_if(w_i(x^\star),y_i)+\frac12\nrm{x^\star-z_{i-1}}^2 .
\]

\emph{(2) From the model to $P$ at $u_i$.} Apply \cref{lem:model} with $v=v_i$, $\tilde y=y_i$ and $x=u_i$. Since $u_i-v_i=w_i(p_i)-w_i(z_{i-1})=\frac{a_i}{A_i}(p_i-z_{i-1})$,
\[
A_iP(u_i)\ \le\ A_if(u_i,y_i)+\frac{\gamma_\chi a_i^2}{A_i}\nrm{p_i-z_{i-1}}^2+2A_i\delta\ \le\ A_if(u_i,y_i)+\frac\theta4\nrm{p_i-z_{i-1}}^2+2A_i\delta ,
\]
where we used $\frac{a_i^2}{A_i}=\frac{\theta}{4\Gamma}$ and $\gamma_\chi\le\Gamma$.

\emph{(3) The right-hand side at $x^\star$.} By convexity of $f(\cdot,y_i)$ and $f(\cdot,y_i)\le P$,
\[
A_if(w_i(x^\star),y_i)\ \le\ A_{i-1}f(u_{i-1},y_i)+a_if(x^\star,y_i)\ \le\ A_{i-1}P(u_{i-1})+a_iP^\star .
\]

\emph{Combining.} Use (2) to replace $A_if(u_i,y_i)$ in (1) by the smaller quantity $A_iP(u_i)-\frac\theta4\nrm{p_i-z_{i-1}}^2-2A_i\delta$, and use (3) on the right-hand side of (1):
\[
A_iP(u_i)+\frac\theta4\nrm{p_i-z_{i-1}}^2+\frac12\Big(1+\frac{\theta\mu}{4\Gamma}\Big)\nrm{z_i-x^\star}^2\ \le\ A_{i-1}P(u_{i-1})+a_iP^\star+\frac12\nrm{z_{i-1}-x^\star}^2+2A_i\delta .
\]
Subtracting $A_iP^\star=A_{i-1}P^\star+a_iP^\star$ from both sides gives \eqref{eq:potential}.

\emph{Consequence.} Drop the nonnegative extra terms on the left of \eqref{eq:potential} and sum over $i=1,\ldots,K$. Using \cref{lem:weights} and $\delta=c_\delta E=\frac{\bar\chi E}{4K^3}$,
\[
V_K\ \le\ V_0+2\delta\sum_{i=1}^KA_i\ \le\ V_0+2\cdot\frac{\bar\chi E}{4K^3}\cdot\frac{\theta K^3}{4\Gamma}=V_0+\frac{\theta E}{8\mu},
\]
where $\Gamma=\mu \bar\chi$. Since $A_0=0$ and $z_0=w$, strong convexity of $P$ gives $V_0=\frac12\nrm{w-x^\star}^2\le\frac E\mu$; with $\theta<1$ we get $V_K\le\frac{9E}{8\mu}$. Finally, $A_K(P(u_K)-P^\star)\le V_K$ and $A_K\ge\frac4\mu$ by \cref{lem:weights}, so
\[
P(u_K)-P^\star\ \le\ \frac{V_K}{A_K}\ \le\ \frac{9E}{8\mu}\cdot\frac\mu4=\frac{9}{32}E .
\]
\end{proof}

\subsection{Proof of \texorpdfstring{\cref{thm:cross}}{the theorem}}

Put $\Lambda=1+\log\frac{B\bar\chi}\eps$. Since $\eps<B$ and $\bar\chi\ge1$, both $\log\frac B\eps$ and $\log \bar\chi$ are at most $\log\frac{B\bar\chi}\eps$.

\emph{Initialization.} Fix $y_0\in\cY$ and apply \cref{lem:init} to $f(\cdot,y_0)$ with $E_{\rm target}=B$. The output $w_0$ satisfies $f(w_0,y_0)-\phi(y_0)\le B/2$. Moreover, for a supergradient $g_0$ of $f(w_0,\cdot)$ at $y_0$ with $\nrm{g_0}\le\My$, concavity gives $f(w_0,y)\le f(w_0,y_0)+\ip{g_0}{y-y_0}\le f(w_0,y_0)+B$ for all $y\in\cY$, so $P(w_0)-f(w_0,y_0)\le B$. Since also $\phi(y_0)\le\max_\cY\phi=P^\star$,
\[
P(w_0)-P^\star=\big[P(w_0)-f(w_0,y_0)\big]+\big[f(w_0,y_0)-\phi(y_0)\big]+\big[\phi(y_0)-P^\star\big]\ \le\ B+\frac B2+0\ \le\ 2B .
\]
This costs $O(I_0)$ gradients and no calls to $\Oy$.

\emph{Outer loop.} Let $E_0=2B$ and $E_s=2^{-s}E_0$. Block $s=1,2,\ldots$ starts from $w_{s-1}$ with $P(w_{s-1})-P^\star\le E_{s-1}$, uses $E=E_{s-1}$ and $\delta=c_\delta E_{s-1}$, and returns $w_s=u_K$. By \cref{lem:potential}, $P(w_s)-P^\star\le\frac9{32}E_{s-1}\le E_s$, so the invariant is maintained. \Cref{alg:cross} stops after the first block $S$ with $E_S\le\eps/(6\bar\chi^2)$, that is,
\[
S=\Big\lceil\log_2\frac{12B\bar\chi^2}\eps\Big\rceil=O(\Lambda),
\]
and by minimality of $S$ we also have $E_S>\eps/(12\bar\chi^2)$. Finally, it computes $\hat y$ with $P(\hat x)-f(\hat x,\hat y)\le c_\delta E_S$ at $\hat x=w_S$.

\emph{Accuracy.} By \cref{lem:gapcross} with $E=E_S$ and $\delta=c_\delta E_S$,
\[
\Gap(\hat x,\hat y)\ \le\ E_S\big(1+2\bar\chi c_\delta+4\bar\chi\chi\big).
\]
Each of the three terms in parentheses is at most a multiple of $\bar\chi^2$. First, $1\le \bar\chi^2$. Second, $K\ge8\sqrt{\bar\chi/\theta}\ge8\sqrt{\bar\chi}$ gives $c_\delta=\frac{\bar\chi}{4K^3}\le\frac{1}{2048\sqrt{\bar\chi}}$, so $2\bar\chi c_\delta\le\frac{\sqrt{\bar\chi}}{1024}\le \bar\chi^2$. Third, $\chi\le \bar\chi$ gives $4\bar\chi\chi\le4\bar\chi^2$. Hence $\Gap(\hat x,\hat y)\le6\bar\chi^2E_S\le\eps$.

\emph{Gradient count.} Each outer block consists of $K$ inner runs of $k$ steps, one gradient each, so it uses $Kk$ calls to $\Ox$. Since $\theta\ge\frac12$, $K\le8\sqrt{2\bar\chi}+1=O(\sqrt{\bar\chi})$, and $k=\lceil2\sqrt{1+\kappa/(4\bar\chi)}\,\rceil=O(\sqrt{1+\kappa/\bar\chi})$. Hence
\[
Kk=O\Big(\sqrt{\bar\chi}\sqrt{1+\kappa/\bar\chi}\Big)=O\big(\sqrt{\bar\chi+\kappa}\big)=O\big(\sqrt{\kappa+\chi}\big),
\]
where we used $\kappa\ge1$. With $S=O(\Lambda)$ blocks and the initialization, $\Tx=I_0+O(\sqrt{\kappa+\chi}\,\Lambda)$.

\emph{Query count in $y$.} The algorithm solves $SK+1=O(\sqrt{\bar\chi}\,\Lambda)$ problems in $y$, each to accuracy at least
\[
\delta_{\min}=c_\delta E_S\ >\ \frac{c_\delta\,\eps}{12\bar\chi^2}.
\]
Since $K\le8\sqrt{2\bar\chi}+1\le13\sqrt{\bar\chi}$, we have $\frac1{c_\delta}=\frac{4K^3}{\bar\chi}\le4\cdot13^3\sqrt{\bar\chi}=8788\sqrt{\bar\chi}$, so $\frac{B}{\delta_{\min}}\le\frac{12\cdot8788\,B\bar\chi^{5/2}}{\eps}$ and $\log\frac{B}{\delta_{\min}}=O(\Lambda)$. By \cref{lem:cog}, each problem in $y$ is solved by the center-of-gravity method with $O(m\Lambda)$ calls to $\Oy$. Therefore
\[
\Ty=O\big(\sqrt{\bar\chi}\,\Lambda\big)\cdot O(m\Lambda)=O\big(m\sqrt{1+\chi}\,\Lambda^2\big).
\]
If instead a solver returns a $\delta$-maximizer after $N_y(\delta)$ calls, with $N_y$ nonincreasing, the same count gives $\Ty=O\big(\sqrt{1+\chi}\,\Lambda\cdot N_y(\eps/(Cr^{5/2}))\big)$ with $C=12\cdot8788$. \hfill$\square$

\section{Experimental details and additional results}\label{app:experiments}

Every number in \cref{sec:experiments} is computed from the raw logs of the runs described below.

\subsection{Problems}\label{app:exp-problems}

All problems have the form
\begin{equation}\label{eq:exp-form}
f(x,y)=\sum_{j=0}^{J} w_j(y)\,L_j(x)+\frac\mu2\nrm{x}^2+\ip{c}{y}-\frac\nu2\nrm{y}^2,
\end{equation}
with convex losses $L_j$, weights $w_j(y)$ that are affine in $y$ and nonnegative on $\cY$, and $\cX=\RR^n$. One call to $\Ox$ returns $\sum_jw_j(y)\nabla L_j(x)+\mu x$ (one weighted backward pass over the data); one call to $\Oy$ returns $f(x,y)$ and $\nabla_y f(x,y)$, i.e., the vector $(L_j(x))_j$ (one forward pass). In our NumPy implementation both calls cost one pass over the data; the measured wall-clock ratio of a call to $\Oy$ to a call to $\Ox$ is between $0.6$ and $1.9$ across the four real problems (in deep models a forward pass typically costs about a third of a gradient). We therefore report $\Tx$ and $\Ty$ separately, and weighted costs $\Tx+r\Ty$ over a range of $r$. We set $\mu=L_{\max}/(\kappa-1)$, where $L_{\max}$ is a certified upper bound on the smoothness constant of $f(\cdot,y)$ over $y\in\cY$ (power iteration on the per-group data matrices), so $\kappa$ is the condition number seen by the methods. The strength of the ridge term is thus the knob that sets $\kappa$; $\kappa=10^4$ corresponds to $\mu\approx10^{-5}$ on unit-norm features, a typical amount of regularization.

\paragraph{Group DRO (Adult; 20~Newsgroups).}
$L_j$ is the mean logistic loss of group $j$, $w(y)=y\in\Delta_G$, and $m=G-1$ (the simplex is parametrized in its affine hull by an orthonormal basis of $\mathbf 1^\perp$, so that $\nrm{\cdot}$ is preserved).
\emph{Adult} \citep{kohavi1996scaling}: $45{,}222$ examples after removing missing values; $n=105$ (standardized numerical features, one-hot categorical features, intercept); $G=4$ groups given by sex $\times$ \{White, non-White\}.
\emph{20~Newsgroups}: binary task \texttt{sci.*} vs.\ \texttt{talk.*} ($7{,}205$ documents); TF-IDF with $2\cdot10^4$ terms plus intercept, rows normalized; the $G=8$ groups are the eight newsgroups, so $m=7$ and $n=20{,}001>N$.

\paragraph{Learning with a few constraints (Adult).}
$L_0$ is the mean logistic loss over all examples and $L_j$, $j=1,\dots,4$, is the mean logistic loss over the \emph{positive} examples of group $j$ (a convex surrogate of the group's false-negative rate). The constraints $L_j(x)\le\tau$ are imposed with $\tau$ equal to $0.9$ times the average of $L_j$ at the unconstrained ridge solution, so that several constraints are active. The Lagrangian is \eqref{eq:exp-form} with $w(y)=(1,y)$, $c=-\tau\mathbf1$, $\cY=[0,\Lambda]^4$, $\Lambda=4$.

\paragraph{Neyman--Pearson classification (20~Newsgroups, $m=1$).}
All $18{,}846$ documents; positives are the four \texttt{sci.*} groups. We minimize the logistic loss on negatives subject to the logistic loss on positives being at most $0.7$ times its value at the unconstrained solution; $\cY=[0,4]$, $n=20{,}001$.

\paragraph{Synthetic group DRO (scaling experiments).}
$L_j(x)=\frac12\nrm{A_jx-b_j}^2$, $j=0,\dots,m$, where $A_j=G_jS/\nrm{G_jS}$ with a Gaussian $G_j\in\RR^{200\times400}$, a shared scaling $S=\operatorname{diag}(i^{-1})$ (so that every $f(\cdot,y)$ is genuinely ill-conditioned) and $b_j=A_jc_j+0.1\xi_j$ with random centers $c_j$; $L_{\max}=1+\mu$. Three independent instances (seeds) per configuration.

\paragraph{Group DRO over all 20 newsgroups ($m=19$).}
All $18{,}846$ documents; the label is technology/science/recreation (\texttt{comp.*}, \texttt{sci.*}, \texttt{rec.*}, \texttt{misc.forsale}) versus society/politics/religion (\texttt{talk.*}, \texttt{alt.atheism}, \texttt{soc.religion.christian}); the $G=20$ groups are the newsgroups, so $m=19$ and $n=20{,}001$. The same TF-IDF features as above.

\paragraph{Weak regularization.}
For the experiments of \cref{fig:weak} the ridge parameter is decreased so that $\kappa$ ranges over $10^4$--$10^7$ ($10^8$ on the synthetic instance), i.e., $\mu/L_{\max}$ between $10^{-4}$ and $10^{-7}$. On 20~Newsgroups $n>N$, so the data term is singular and $\mu$ is the actual strong convexity constant.

\paragraph{Group DRO with an $\ell_1$ penalty (nonsmooth dual).}
On Adult with the four groups above, $f(x,y)=\sum_gw_gL_g(x)+\frac\mu2\nrm x^2-\rho\nrm{w-p}_1$, where $p$ is the vector of group proportions and $\rho=0.1$. The dual function is then nonsmooth; at the solution one group weight sits exactly at its kink $w_g=p_g$ and two are at $0$.

\paragraph{$\chi^2$-regularized group DRO (\cref{sec:strong}).}
On 20~Newsgroups (\texttt{sci.*} vs.\ \texttt{talk.*}) with $G\in\{2,4,8\}$ groups (the two topics, four pairs of newsgroups, the eight newsgroups), $f=\sum_gw_gL_g(x)+\frac\mu2\nrm x^2-\frac\nu2\nrm{w-p}^2$ with $\nu=1$ and $\kappa=10^3$. Since $\nrm{\nabla L_g}\le\max_i\nrm{a_i}=1$, the coupling constant is certified as $\ell=\sqrt G$, so $\chi=\ell^2/(\mu\nu)=G/\mu$ grows with the number of groups (see the discussion after \cref{cor:combined}); it equals $1.5\cdot10^4$, $3.1\cdot10^4$ and $6.2\cdot10^4$ for $G=2,4,8$.

\paragraph{Strongly concave problem with bounded coupling (\cref{sec:strong}).}
$f(x,y)=\frac12x\T Hx-\ip{h}{x}+\ip{y}{Cx}-\frac\nu2\nrm y^2$ on $\cY=[-1,1]^m$, $n=400$, $\operatorname{spec}H$ log-uniform in $[\mu,1]$, and $C=\ell Q\T$ with $Q$ having orthonormal columns. Hence $\nrm{\nabla_xf(x,y)-\nabla_xf(x,y')}\le\ell\nrm{y-y'}$ with the same $\ell$ for every $m$; we take $\kappa=10^4$ and $\chi=\ell^2/(\mu\nu)=1$.

\subsection{Methods and implementation}\label{app:exp-methods}

\paragraph{Certificate transport.}
\Cref{alg:ct} exactly as stated, with the default $K_0=\lceil\sqrt{8\kappa/m}\,\rceil$ unless a sweep over $K$ is reported. The lower model is initialized on the slice $y_0=\cg(\cY)$ by accelerated gradient descent stopped by the gradient certificate $f(x,y_0)\ge f(u,y_0)-\frac{\nrm{g}^2}{2\mu}+\frac\mu2\nrm{x-(u-g/\mu)}^2$, which is a valid model \eqref{eq:lowermodel} and replaces the restarts of \cref{lem:init} (whose only role is to handle an unknown $R_x$). $a=\max_\cY U$ and the weights $\lambda$ of the output $\bar x$ (\cref{lem:recovery}) are obtained from one LP whose dual variables are $\lambda$.
We also report a variant, \emph{CT with geometric steps}, which differs only in the step sizes of \cref{alg:block}: instead of $a_i=(i+1)\gamma/(2L)$ it takes the largest $a_i$ admitted by the proof of \cref{lem:block}, i.e., the root of $La_i^2=\Gamma_{i-1}A_i$ (the step rule of accelerated gradient descent for strongly convex functions). The proof of \cref{lem:block} uses the step sizes only through $La_i^2\le\Gamma_{i-1}A_i$, so \cref{lem:block}, \cref{cor:transport} and \cref{thm:general} hold verbatim; the only difference is that $A_K$ becomes larger (exponentially in $K/\sqrt\kappa$ once $K\gtrsim\sqrt\kappa$), which makes $\vartheta$ closer to $1$.

\paragraph{Practical certificate transport.}
In addition to geometric steps, \emph{CT (practical)} uses two further modifications, neither of which changes \cref{thm:general}. (a)~\emph{Prior curvature $\sigma=\mu$.} The initial model has curvature $\mu$, and \cref{cor:transport} maps a model of curvature $\sigma$ to one of curvature $(\sigma+\mu A)/(1+A)$, which equals $\mu$ when $\sigma=\mu$. Hence one can run \cref{alg:block} with $\gamma=\mu$ throughout ($\sigma=\mu/2$ is needed only in \cref{lem:init}), which doubles $A$ and therefore increases $\vartheta$. (b)~\emph{Direct certificate.} After the block, one extra call $g_u=\Ox(u,y_c)$ gives the model $f(x,y_c)\ge v-\frac{\nrm{g_u}^2}{2\mu}+\frac\mu2\nrm{x-(u-g_u/\mu)}^2$ on the queried slice itself. The algorithm keeps whichever of this model and the transported one has the larger constant, and only if it exceeds the current $b$. Thus $b$ never decreases and still grows by at least $\vartheta(v-b)$ after every transport, so the counting in \cref{app:minmax-general} is unchanged; each query costs $K+1$ calls to $\Ox$. When the block has essentially solved the slice, (b) makes CT jump to the slice like a nested scheme; when it has not, the transported bound is used. The default block length $K_0$ is unchanged.

\paragraph{Nested cutting planes.}
The textbook scheme of \cref{eq:nested}: the center-of-gravity method on $\phi$; at each center $y_c$, accelerated gradient descent on $f(\cdot,y_c)$ stopped as soon as $\nrm{\nabla_xf}^2/(2\mu)\le\delta$, which certifies $f(\tilde x,y_c)-\phi(y_c)\le\delta$; the cut is the plane of $f(\tilde x,\cdot)$ and the lower bound is $f(\tilde x,y_c)-\delta$ (\cref{lem:cog}). \emph{Cold start}: every inner run starts from $x=0$ and $\delta=\eps/2$, as in the analysis of \citet{gladin2023solving}. \emph{Warm start}: the inner run starts from the previous inner solution (same $\delta$); this heuristic has no better worst-case guarantee but is what one would implement in practice. \emph{Warm, adaptive}: in addition $\delta=\max\{\eps/2,\,0.1\cdot\text{current certified width}\}$.
All nested variants use the same center-of-gravity routine, LPs and output recovery as CT.

\paragraph{Level-localizer ablation.}
This baseline is the warm-started nested scheme with adaptive inner accuracy, but with the localizer and epochs of \cref{alg:ct}. Queries are centers of gravity of $\{y:U(y)\ge h\}$ with $h=\beta+W/2$ updated in epochs as in \cref{alg:ct}, and each query is an inner accelerated run from the previous inner point stopped at accuracy $\max\{\eps/2,\,c\,(a-\beta)\}$. Because $U\ge\phi$, a cut never removes a point with $\phi\ge h$, which repairs the failure mode of the plain adaptive scheme (\cref{app:exp-more}); the certified width is valid as for all nested schemes. This is an inexact level method in the sense of \citet{lemarechal1995new}. We are not aware of a bound of order $\sqrt{m\kappa}\log\frac1\eps$ for it: the cost of each warm-started inner run grows with the distance between consecutive slice solutions. We ran $c\in\{0.03,0.1,0.3,0.5,1\}$ on 20~Newsgroups ($m=7,19$) and on the synthetic instances ($m=4,16$, seed $0$), $c\in\{0.1,0.3,0.5\}$ on the real problems at $\kappa=10^4$, and the default $c=0.1$ elsewhere, and we report the best $c$ in hindsight. Large fractions make each inner solve cheap but the lower bound $\beta$ loose. With $c=1$ the method failed to certify $\eps$ on the synthetic instance with $m=4$ for every $\kappa$, on the synthetic instance with $m=16$ for $\kappa\ge10^6$, and on 20News with $20$ groups at $\kappa=10^7$ (budget: $2\cdot10^4$ calls to $\Oy$ or $5$ hours), and it needed thousands of $y$-queries on the problems with $m\ge16$ when it did converge.

\paragraph{AGD on the primal function (\cref{sec:strong}).}
Nesterov's method for strongly convex functions on $P$ with step $1/L_P$, $L_P=L+2\ell^2/\nu$, and the inexact gradient $\nabla_xf(v,\tilde y)$ at a $\delta$-maximizer $\tilde y$ of $f(v,\cdot)$ with $\delta=\eps/(4\sqrt{L_P/\mu})$. By \cref{lem:model} this is a $(2\delta,L_P)$-inexact oracle \citep{devolder2014first}. It uses one call to $\Ox$ per step and one $y$-subproblem per step. For both this baseline and \cref{alg:cross}, the $y$-subproblems of the strongly concave problems are solved by projected gradient ascent with step $1/\nu$ (\cref{thm:cross} allows any $\delta$-maximizer), so the number of calls to $\Oy$ is proportional to the number of $y$-subproblems.

\paragraph{Single-loop baselines.}
Extragradient \citep{korpelevich1976extragradient} on the joint variable with Euclidean projection onto $\cY$ (two calls to each oracle per iteration) and, for group DRO, simultaneous gradient descent in $x$ with exponentiated-gradient ascent on the group weights \citep[deterministic version of][]{sagawa2020distributionally}. Both have separate step sizes for the two blocks, which we tune over a grid ($\eta_x\in\{0.5,1\}/L_{\max}$, $\eta_y\in\{0.1,1,10,100\}$) and report the best configuration \emph{in hindsight}, i.e., the one reaching the target accuracy with the fewest calls to $\Ox$; both the last and the averaged iterate are evaluated and the better one is reported. These choices favor the baselines. CT and the nested schemes have no tuned parameters.

\paragraph{Centers of gravity.}
For $m\le4$ we compute the exact centroid of the localizer (vertex enumeration with Qhull and triangulation). For $m\ge5$ we use the approximate centroid given by the average of $\max(4000,150m^2)$ hit-and-run samples with a rounding matrix carried over from the previous call, in the spirit of \citet{bertsimas2004solving}. This affects only the number of iterations, never the correctness of the certificates, which hold for every choice of query points.

\paragraph{Evaluation.}
The reported gap is $\Gap(\hat x,\hat y)=P(\hat x)-\phi(\hat y)$ of each method's output, computed independently of the method and without counting oracle calls: $P(\hat x)$ is computed exactly (a vertex of the simplex or the box, or a projection when $\nu>0$), and $\phi(\hat y)$ is bounded from below by $f(\tilde x,\hat y)-\nrm{\nabla_xf(\tilde x,\hat y)}^2/(2\mu)$ after running Newton-CG on $f(\cdot,\hat y)$ to relative accuracy $10^{-4}$. The reported numbers are therefore rigorous upper bounds on the gap. For CT and the nested schemes the output is $(\bar x,y_b)$ (resp.\ $(\bar x,y_{\rm best})$), exactly as returned by the algorithms; in every run we checked that the true gap never exceeds the certified width.
The evaluation solves for $\phi(\hat y)$ until the gradient bound $\nrm{\nabla_xf(\tilde x,\hat y)}^2/(2\mu)$ is below $10^{-4}$ times the gap being measured (a relative accuracy \emph{of the gap}, not of $\phi$), so the reported gap may exceed the true gap by at most this relative amount; in $24$ of the $4{,}068$ recorded points (from $12$ runs) this made the reported gap exceed the certified width, by at most $6.4\cdot10^{-5}$ in relative terms. We re-ran these $12$ runs and re-evaluated all their outputs with relative accuracy $10^{-7}$ or better; the true gap was then below the certified width everywhere.

\paragraph{Compute.}
All experiments ran on a single 22-core CPU machine, one process and one BLAS thread per run; the whole suite (about $900$ runs) takes several hours on this machine and is dominated by the tuned single-loop baselines and by hit-and-run centers for $m\ge16$.

\subsection{Additional results}\label{app:exp-more}

\begin{table}[h]
\centering
\caption{Weak regularization ($\eps=10^{-6}$): calls $\Tx$\,/\,$\Ty$ of CT, of the level baseline with its default fraction $c=0.1$ and with its best fraction in hindsight (value of $c$ in parentheses), and of the adaptive and warm-started nested schemes. Bold: fewest calls in each row (separately for $\Tx$ and $\Ty$). ``--'': not run or did not certify $\eps$ within the budget.}
\label{tab:weak}
\scriptsize
\setlength{\tabcolsep}{3pt}
\begin{tabular}{llccccc}
\toprule
problem & $\kappa$ & CT & level, $c=0.1$ & level, best $c$ & nested, adaptive & nested, warm\\
\midrule
20News, 20 groups ($m=19$) & $10^{4}$ & 8{,}735\,/\,112 & 16{,}523\,/\,\textbf{98} & \textbf{7{,}644}\,/\,111 (0.5) & 17{,}941\,/\,122 & 36{,}684\,/\,125\\
 & $10^{5}$ & \textbf{25{,}933}\,/\,104 & 57{,}399\,/\,\textbf{98} & 27{,}907\,/\,111 (0.5) & 58{,}557\,/\,113 & 100{,}665\,/\,110\\
 & $10^{6}$ & \textbf{79{,}521}\,/\,100 & 168{,}809\,/\,\textbf{81} & 83{,}557\,/\,139 (0.5) & 182{,}646\,/\,107 & 332{,}526\,/\,111\\
 & $10^{7}$ & \textbf{228{,}277}\,/\,88 & 401{,}969\,/\,\textbf{61} & 267{,}176\,/\,1{,}082 (0.5) & 465{,}067\,/\,125 & 1{,}172{,}507\,/\,128\\
\addlinespace
20News, 8 groups ($m=7$) & $10^{4}$ & 6{,}131\,/\,45 & 8{,}367\,/\,46 & \textbf{3{,}924}\,/\,47 (0.5) & 7{,}769\,/\,46 & 13{,}488\,/\,\textbf{41}\\
 & $10^{5}$ & 19{,}150\,/\,44 & 27{,}889\,/\,41 & \textbf{13{,}453}\,/\,42 (0.5) & 25{,}186\,/\,43 & 39{,}142\,/\,\textbf{40}\\
 & $10^{6}$ & 58{,}891\,/\,42 & 79{,}268\,/\,\textbf{35} & \textbf{43{,}979}\,/\,413 (1) & 77{,}770\,/\,39 & 113{,}717\,/\,37\\
 & $10^{7}$ & 170{,}604\,/\,37 & 206{,}029\,/\,\textbf{28} & \textbf{116{,}934}\,/\,49 (0.5) & 187{,}866\,/\,38 & 355{,}585\,/\,41\\
\addlinespace
Adult, 4 groups ($m=3$) & $10^{4}$ & 3{,}601\,/\,\textbf{16} & -- & -- & \textbf{2{,}655}\,/\,43 & 6{,}774\,/\,43\\
 & $10^{5}$ & 11{,}356\,/\,\textbf{16} & 7{,}651\,/\,19 & 7{,}651\,/\,19 (0.1) & \textbf{7{,}400}\,/\,43 & 20{,}590\,/\,43\\
 & $10^{6}$ & 37{,}732\,/\,\textbf{17} & 22{,}354\,/\,19 & 22{,}354\,/\,19 (0.1) & \textbf{19{,}905}\,/\,44 & 66{,}580\,/\,44\\
 & $10^{7}$ & 118{,}464\,/\,\textbf{17} & 68{,}531\,/\,20 & 68{,}531\,/\,20 (0.1) & \textbf{58{,}671}\,/\,44 & 251{,}010\,/\,44\\
\addlinespace
20News NP ($m=1$) & $10^{4}$ & 3{,}178\,/\,\textbf{13} & \textbf{1{,}378}\,/\,16 & \textbf{1{,}378}\,/\,16 (0.1) & 2{,}410\,/\,16 & 6{,}828\,/\,15\\
 & $10^{5}$ & 10{,}016\,/\,\textbf{13} & \textbf{4{,}760}\,/\,16 & \textbf{4{,}760}\,/\,16 (0.1) & -- & 23{,}144\,/\,16\\
 & $10^{6}$ & 31{,}568\,/\,\textbf{13} & \textbf{21{,}064}\,/\,17 & \textbf{21{,}064}\,/\,17 (0.1) & 24{,}615\,/\,19 & 75{,}881\,/\,17\\
 & $10^{7}$ & 99{,}526\,/\,\textbf{13} & \textbf{60{,}839}\,/\,17 & \textbf{60{,}839}\,/\,17 (0.1) & -- & 239{,}422\,/\,18\\
\addlinespace
Synthetic ($m=16$) & $10^{4}$ & 9{,}248\,/\,130 & 6{,}992\,/\,\textbf{57} & \textbf{4{,}177}\,/\,66 (0.3) & 26{,}285\,/\,222 & 71{,}113\,/\,224\\
 & $10^{5}$ & 29{,}284\,/\,131 & 44{,}534\,/\,\textbf{81} & \textbf{19{,}329}\,/\,164 (0.5) & 91{,}504\,/\,178 & 193{,}446\,/\,175\\
 & $10^{6}$ & \textbf{85{,}893}\,/\,121 & 219{,}211\,/\,93 & 143{,}954\,/\,\textbf{85} (0.3) & 364{,}163\,/\,156 & 626{,}619\,/\,153\\
 & $10^{7}$ & \textbf{277{,}572}\,/\,123 & 976{,}732\,/\,\textbf{101} & 647{,}344\,/\,214 (0.5) & 1{,}364{,}564\,/\,141 & 2{,}136{,}345\,/\,139\\
\addlinespace
Synthetic ($m=4$) & $10^{4}$ & 4{,}568\,/\,33 & 4{,}235\,/\,\textbf{30} & \textbf{3{,}196}\,/\,35 (0.3) & 5{,}610\,/\,34 & 12{,}075\,/\,32\\
 & $10^{5}$ & 15{,}869\,/\,36 & 18{,}392\,/\,\textbf{31} & \textbf{9{,}717}\,/\,44 (0.5) & 23{,}322\,/\,35 & 43{,}280\,/\,35\\
 & $10^{6}$ & 47{,}935\,/\,34 & 76{,}521\,/\,\textbf{32} & \textbf{30{,}944}\,/\,35 (0.5) & 97{,}874\,/\,36 & 163{,}584\,/\,36\\
 & $10^{7}$ & 135{,}480\,/\,\textbf{30} & 304{,}671\,/\,31 & \textbf{133{,}266}\,/\,32 (0.5) & 360{,}539\,/\,33 & 551{,}295\,/\,32\\
 & $10^{8}$ & \textbf{433{,}824}\,/\,\textbf{30} & 1{,}121{,}064\,/\,\textbf{30} & 713{,}778\,/\,35 (0.5) & 1{,}411{,}130\,/\,35 & 2{,}013{,}173\,/\,34\\
\bottomrule
\end{tabular}
\end{table}

\begin{table}[h]
\centering
\caption{The level baseline with a fixed accuracy fraction $c$, over all (problem, $\kappa$) cells of \cref{tab:weak} where it was run: geometric mean of $\Tx(\text{level}_c)/\Tx(\text{CT})$ over the cells where it certified $\eps$, and the number of cells where it did not. Bold: best value in each row.}
\label{tab:levelfrac}
\small
\begin{tabular}{lccccc}
\toprule
$c$ & 0.03 & 0.1 & 0.3 & 0.5 & 1\\
\midrule
geometric mean of the ratio & 2.01 & 1.21 & 1.09 & \textbf{0.91} & 1.07\\
failures / cells & \textbf{0/13} & 1/25 & \textbf{0/17} & 1/17 & 8/17\\
\bottomrule
\end{tabular}
\end{table}

\begin{table}[h]
\centering
\caption{Calls $\Tx$\,/\,$\Ty$ until the true gap is $\le10^{-7}$ on the real problems ($\kappa=10^4$). Tuned methods: best configuration in hindsight (the level baseline over $c\in\{0.1,0.3,0.5\}$). Single-loop methods are evaluated on a geometric grid of ratio $1.04$. Bold: fewest calls in each column (separately for $\Tx$ and $\Ty$). $^\dagger$No guarantee of the order of \cref{thm:general}.}
\label{tab:real}
\scriptsize
\setlength{\tabcolsep}{3pt}
\begin{tabular}{lccccc}
\toprule
 & NP, 20News & DRO, Adult & constr., Adult & DRO, 20News & DRO$+\ell_1$, Adult\\
 & $m=1$ & $m=3$ & $m=4$ & $m=7$ & $m=3$\\
\midrule
CT (practical) & 3{,}462\,/\,\textbf{14} & 3{,}931\,/\,\textbf{18} & 3{,}817\,/\,\textbf{27} & 6{,}995\,/\,\textbf{53} & 2{,}670\,/\,\textbf{17}\\
CT, steps of \cref{alg:block} & 6{,}563\,/\,25 & 16{,}051\,/\,92 & 17{,}140\,/\,121 & 8{,}977\,/\,72 & 10{,}199\,/\,63\\
Nested, cold start & 11{,}575\,/\,16 & 33{,}127\,/\,49 & 50{,}231\,/\,63 & 46{,}941\,/\,56 & 26{,}933\,/\,39\\
Nested, warm start & 8{,}687\,/\,16 & 10{,}158\,/\,51 & 22{,}808\,/\,65 & 20{,}862\,/\,55 & 10{,}279\,/\,38\\
Nested, warm+adaptive$^\dagger$ & 2{,}810\,/\,17 & \textbf{3{,}262}\,/\,51 & 6{,}939\,/\,65 & 9{,}691\,/\,55 & 4{,}289\,/\,40\\
Level, best $c$$^\dagger$ & \textbf{1{,}603}\,/\,16 ($c=0.1$) & 3{,}344\,/\,40 ($c=0.3$) & \textbf{2{,}190}\,/\,40 ($c=0.3$) & \textbf{4{,}562}\,/\,54 ($c=0.5$) & \textbf{1{,}523}\,/\,18 ($c=0.1$)\\
Extragradient (tuned) & 74{,}154\,/\,74{,}154 & 54{,}008\,/\,54{,}008 & 66{,}718\,/\,66{,}718 & 91{,}606\,/\,91{,}606 & $>$100{,}000\\
GDA-MW (tuned) & -- & 29{,}853\,/\,29{,}853 & -- & 54{,}433\,/\,54{,}433 & --\\
\bottomrule
\end{tabular}
\end{table}

\paragraph{Cuts and transports.}
With few groups almost every query of CT raises the lower bound: $f$ is linear in $y$ and $f(\cdot,y)$ is strongly convex, so $\phi$ is differentiable (Danskin's theorem), a few planes make the upper model accurate, and the certified width is dominated by $b$. With many groups the localizer matters: on 20News DRO with $m=7$ about half of the queries of CT are cuts, and on the 20-group problem $28$--$41$ of about $100$. Without the direct certificate (``CT, steps of \cref{alg:block}'' and ``geometric steps'' in the figures) almost all queries are transports, and larger blocks reduce $\Tx$ up to $K\approx(1.6\text{--}3.6)K_0$ (\cref{fig:ksweep}).

\paragraph{Wall-clock time.}
All runs shared one machine with a varying load, so we compare oracle calls, not time. The share of wall-clock time that CT spends inside the two oracles is $82$--$100\%$ for $m\le4$ with exact centers, $66$--$92\%$ for $m=7$, and $17$--$89\%$ for $m\in\{16,19\}$, where hit-and-run centers are relatively expensive when $\kappa$ (and hence the cost per query) is small. Computing implementable centers more efficiently \citep{bertsimas2004solving,vaidya1996new} is therefore what matters in practice for $m\gtrsim10$.

\paragraph{Seeds of the approximate centers.}
On 20News DRO ($m=7$, hit-and-run centers) three seeds of the center oracle changed the number of calls to $\Ox$ of CT by at most $2\%$ and of the nested and level schemes by at most $6\%$.

\begin{table}[h]
\centering
\caption{Price of the $y$-block for $m=1$ (Neyman--Pearson on 20News, $\eps=10^{-6}$): calls to $\Ox$ of each method divided by the calls of accelerated gradient descent on $f(\cdot,y^\star)$ with $y^\star$ known in advance (absolute counts of the latter in the last row).}
\label{tab:price}
\small
\begin{tabular}{lcccc}
\toprule
$\kappa$ & $10^2$ & $10^3$ & $10^4$ & $10^5$\\
\midrule
CT (\cref{alg:ct}) & 12.2 & 12.9 & 14.0 & 16.5\\
CT, geometric steps & 7.6 & 8.6 & 9.2 & 11.6\\
Nested, cold start & 14.3 & 15.9 & 17.1 & 21.4\\
Nested, warm start & 9.6 & 11.6 & 12.9 & 15.7\\
\midrule
AGD on $f(\cdot,y^\star)$ (calls) & 50 & 169 & 529 & 1{,}477\\
\bottomrule
\end{tabular}
\end{table}

\begin{figure}[h]
\centering
\includegraphics[width=\linewidth]{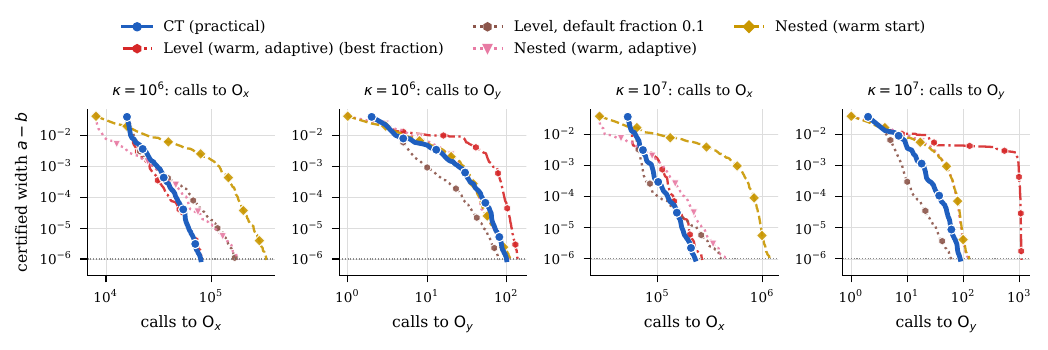}
\caption{Group DRO over the 20 newsgroups ($m=19$, $n=20{,}001$): certified width versus calls to $\Ox$ and to $\Oy$ for $\kappa=10^6$ and $10^7$.}
\label{fig:weakconv}
\end{figure}

\begin{figure}[h]
\centering
\includegraphics[width=\linewidth]{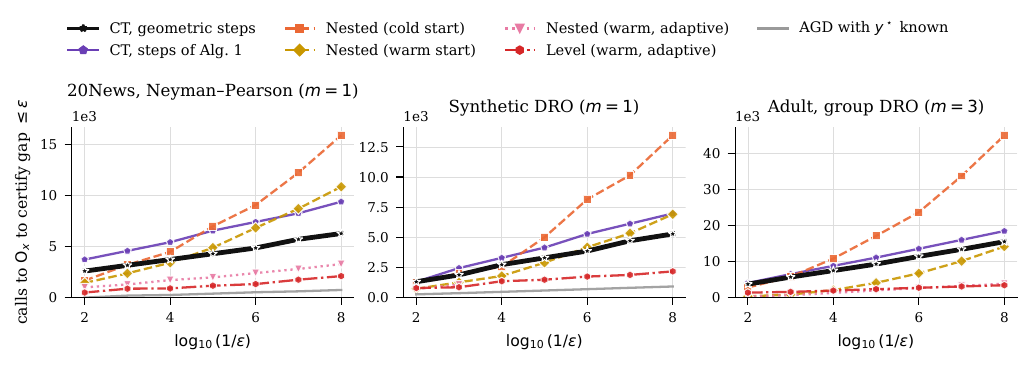}
\caption{Calls to $\Ox$ needed to certify gap $\le\eps$ ($\kappa=10^4$). CT is run once and read off at every $\eps$ (anytime); the nested schemes are rerun for every $\eps$ because their inner accuracy depends on $\eps$. The gray curve is accelerated gradient descent on $f(\cdot,y^\star)$ with $y^\star$ given in advance.}
\label{fig:eps}
\end{figure}

\begin{figure}[h]
\centering
\includegraphics[width=\linewidth]{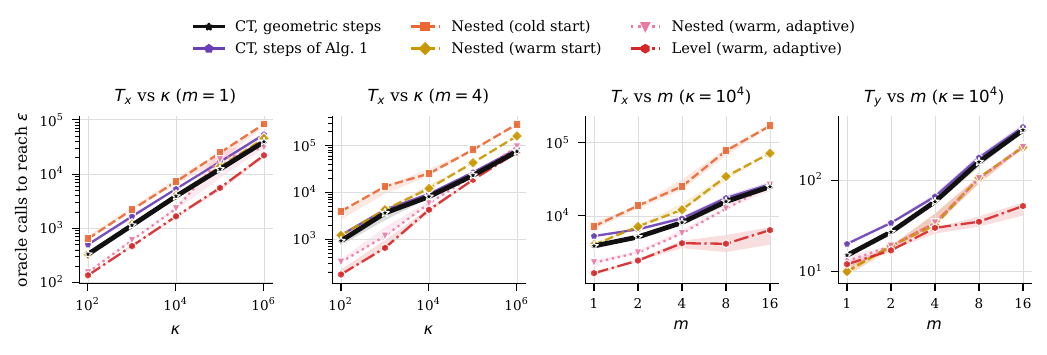}
\caption{Synthetic group DRO: calls to reach $\eps=10^{-6}$ (median over three instances, band: min--max). Left: dependence on $\kappa$. Right: dependence on $m$ at $\kappa=10^4$.}
\label{fig:scaling}
\end{figure}

\begin{figure}[h]
\centering
\includegraphics[width=\linewidth]{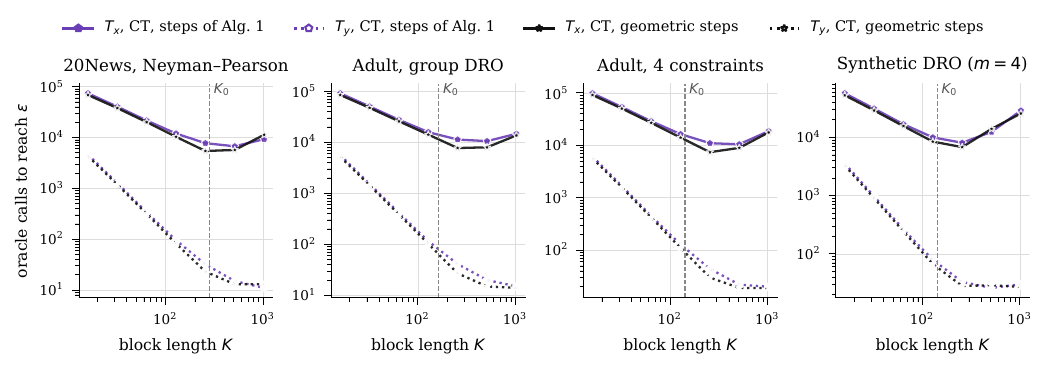}
\caption{Calls to $\Ox$ (solid) and $\Oy$ (dotted) to reach $\eps=10^{-6}$ as a function of the block length $K$, for the step sizes of \cref{alg:block} and for geometric steps ($\kappa=10^4$). The dashed line marks the default $K_0=\lceil\sqrt{8\kappa/m}\,\rceil$. Runs that did not reach $\eps$ within $2\cdot10^4$ calls to $\Oy$ are omitted.}
\label{fig:ksweep}
\end{figure}

\paragraph{Warm starts with adaptive inner accuracy: fast but not robust.}
Making the inner accuracy of the warm-started nested scheme proportional to the current width, $\delta=\max\{\eps/2,0.1(a-\beta)\}$, makes the nested scheme the fastest method on three of our four real problems, but it can break the method. A coarse early cut is the plane of $f(\tilde x,\cdot)$ with an inexact $\tilde x$ and can remove the maximizer of $\phi$ from the localizer. \Cref{lem:cog} then only guarantees $a-\beta\le B(1-e^{-1})^{t/m}+\max_s\delta_s$. This happened in our synthetic experiments: the variant failed to certify the target accuracy in $4$ of the $39$ scaling runs, all with $m=1$ ($4$ of $15$ such runs), and on the instance of \cref{fig:eps} (synthetic, $m=1$) for every target $\eps\le10^{-4}$, and in $1$ of the $12$ runs of the strongly concave experiment. In each failure the certified width froze between $5\cdot10^{-6}$ and $3\cdot10^{-4}$ and did not move during $2\cdot10^4$ further calls to $\Oy$. It never failed on the real problems. CT and the two nested schemes with guarantees certified the target accuracy in every one of their runs.

\section{Detailed discussion of related work}\label{app:related}

\paragraph{Information-based complexity and cutting-plane methods.}
The theory of oracle complexity was laid out in the monograph of \citet{nemirovski1983problem}. It shows that $m$-dimensional convex problems in the high-accuracy regime require, and can be solved with, $\Theta(m\log\frac1\eps)$ first-order queries. The upper bound is attained by the center-of-gravity method \citep{levin1965algorithm,newman1965location}, which rests on Grünbaum's inequality \citep{grunbaum1960partitions}. The method itself is not implementable. The ellipsoid method \citep{shor1977cut,nemirovski1983problem} makes it polynomial-time at the price of $O(m^2\log\frac1\eps)$ queries, while Vaidya's method \citep{vaidya1996new}, approximate centers of gravity based on random walks \citep{bertsimas2004solving} and modern cutting-plane methods \citep{lee2015faster,jiang2020improved} keep the optimal linear dependence on $m$ at polynomial arithmetic cost; the last of these works also treats convex-concave games. Using only function values, \citet{lee2018efficient} need $\Otil(m^2)$ evaluations. Accuracy certificates \citep{nemirovski2010accuracy} make the output of cutting-plane methods verifiable; we use a related, model-based notion of certificate. \citet{rodomanov2023subgradient} proposed a subgradient ellipsoid method that interpolates between the ellipsoid method and subgradient methods; we return to it in \cref{app:open}. For a textbook exposition, see \citet{nesterov2018lectures}.

\paragraph{Problems with a low-dimensional block.}
Problems in which one group of variables has small dimension were studied systematically in a series of works \citep{gladin2021solving,gladin2023solving,gladin2022vaidya,gladin2023algorithm,pasechnyuk2019one}. \citet{gladin2021solving} consider min-min problems that are smooth and strongly convex in one block and low-dimensional in the other: an outer ellipsoid-type method is combined with an inner accelerated method, and the effect of the inexactness of the outer oracle is analyzed. \citet{gladin2023solving} use the same kind of nested scheme with inexact oracles for strongly convex-concave composite saddle-point problems with a low-dimensional block and obtain, in our notation, the bound \eqref{eq:nested}, $\Tx=O\big(m\sqrt\kappa\log\frac B\eps\log\frac{\mu\Rx^2+B}\eps\big)$, up to an additional $\log m$ inside the first logarithm that comes from their outer method. \citet{gladin2022vaidya} extend Vaidya's method to stochastic low-dimensional problems, and \citet{gladin2023algorithm} use a low-dimensional dual problem to build linearly convergent algorithms for constrained Markov decision processes. \citet{pasechnyuk2019one} develop a dichotomy-type method for functions of two variables.

Our work continues this line. We show that the multiplicative cost of nested schemes can be removed with no assumptions on the $y$-block beyond concavity and Lipschitz continuity: the product $m\sqrt\kappa$ is replaced by $\sqrt{m\kappa}$, and the second logarithm disappears. The key difference from nested schemes is that inner problems are never solved to a prescribed accuracy: an inexact answer is either used as a cut or transports the lower model to a new point.

\paragraph{Inexact oracles and models.}
The analysis of inexact oracles goes back to \citet{devolder2014first}; \citet{stonyakin2021inexact} introduce the notion of an inexact model of a function. The method of \cref{sec:strong} relies on exactly this technique: under strong concavity in $y$, the primal function $P$ admits a smooth inexact model built from approximate solutions of the problem in $y$.

\paragraph{Complexity separation in saddle-point problems.}
Gradient sliding \citep{lan2016gradient} separates the numbers of gradient evaluations for the smooth and nonsmooth parts of composite problems. \citet{alkousa2020accelerated} separate the complexities in the blocks $x$ and $y$ for smooth strongly convex-concave problems; \citet{borodich2023optimal} obtain complexity separation for strongly convex-strongly concave composite problems of the form $p(x)+R(x,y)-q(y)$; \citet{kovalev2022accelerated} and \citet{thekumparampil2022lifted} propose optimal methods for bilinear coupling; and \citet{kovalev2024linear} establish lower bounds and optimal algorithms that separate the numbers of gradient evaluations and matrix-vector products for bilinearly coupled problems. \citet{jin2022sharper} obtain sharper rates for separable minimax problems. For smooth convex-concave problems, lower complexity bounds are given by \citet{zhang2022lower}; classical methods include the extragradient method \citep{korpelevich1976extragradient}, mirror-prox \citep{nemirovski2004prox} and smoothing \citep{nesterov2005smooth}.

All of these works assume smoothness in both blocks (or a nonsmooth part with a simple proximal step) and give dimension-free bounds. We instead allow a nonsmooth, merely concave $y$-block and exploit its small dimension. Closer in spirit, \citet{carmon2021thinking} show that minimizing the maximum of $N$ losses admits near-optimal methods based on ball-oracle acceleration \citep{carmon2020acceleration}. Their oracle, however, accesses one loss at a time, whereas our $y$-oracle returns a value and a supergradient for the whole small block at once, so the oracle models differ.

\paragraph{Other oracles for the small block.}
The small dimension of the $y$-block also makes other types of oracles attractive; we discuss them in \cref{app:open}. In particular, a derivative of order $p$ in dimension $m$ has only $m^p$ entries, so tensor methods \citep{nesterov2021implementable}, which attain optimal rates for functions with Lipschitz higher-order derivatives \citep{kovalev2022first,carmon2022optimal}, become affordable for the small block.

\section{Open questions}\label{app:open}
\begin{enumerate}[leftmargin=1.3em,itemsep=2pt,topsep=2pt]
\item \textbf{Implementable centers.} Our analysis uses the exact center of gravity of the localizer and treats its computation as free. Approximate centers of gravity based on random walks \citep{bertsimas2004solving} satisfy Grünbaum's lemma with a worse constant, and we expect our bounds to carry over with only the absolute constants changed. Vaidya's method \citep{vaidya1996new} relies on a different potential, and combining it with certificate transport requires a separate analysis.

\item \textbf{A small block with function values only.} If $\Oy$ returns only $f(x,y)$, a separation oracle for a concave Lipschitz function in dimension $m$ can be emulated with $\Otil(m)$ function values \citep{lee2018efficient}. Since certificate transport needs only the single value $v=f(u_K,y_c)$, one may expect $\Ty=\Otil(m^2\log\frac B\eps)$ with the same $\Tx$. The difficulty is that planes built from approximate supergradients are no longer exact upper bounds on $\phi$, so the bound $a$ must be replaced by an approximate one with controlled error.

\item \textbf{Adaptive block length and adaptive steps in $y$.} \Cref{thm:general} uses a fixed block length $K$. In practice, one can monitor the fraction of iterations that transport the certificate and adjust $K$ on the fly. The center-of-gravity step could also be replaced by a step of the subgradient ellipsoid method \citep{rodomanov2023subgradient}, which produces accuracy certificates \citep{nemirovski2010accuracy} and interpolates between cutting planes and subgradient methods. This would allow a smooth transition from the low-dimensional to the high-dimensional regime in $y$ without switching algorithms.
\end{enumerate}

Finally, the small dimension of the $y$-block also makes higher-order oracles attractive: a derivative of order $p$ in dimension $m$ has only $m^p$ entries, and tensor methods \citep{nesterov2021implementable,kovalev2022first,carmon2022optimal} could replace cutting planes when $f$ is sufficiently smooth in $y$.

\clearpage
\bibliographystyle{plainnat}
\bibliography{references}

@book{nemirovski1983problem,
  author    = {Nemirovsky, A. S. and Yudin, D. B.},
  title     = {Problem Complexity and Method Efficiency in Optimization},
  publisher = {Wiley},
  address   = {New York},
  year      = {1983}
}

@article{levin1965algorithm,
  author  = {Levin, A. Yu.},
  title   = {On an algorithm for the minimization of convex functions},
  journal = {Soviet Mathematics Doklady},
  volume  = {6},
  pages   = {286--290},
  year    = {1965}
}

@article{newman1965location,
  author  = {Newman, Donald J.},
  title   = {Location of the maximum on unimodal surfaces},
  journal = {Journal of the ACM},
  volume  = {12},
  number  = {3},
  pages   = {395--398},
  year    = {1965}
}

@article{shor1977cut,
  author  = {Shor, N. Z.},
  title   = {Cut-off method with space extension in convex programming problems},
  journal = {Cybernetics},
  volume  = {13},
  number  = {1},
  pages   = {94--96},
  year    = {1977}
}

@article{grunbaum1960partitions,
  author  = {Gr{\"u}nbaum, B.},
  title   = {Partitions of mass-distributions and of convex bodies by hyperplanes},
  journal = {Pacific Journal of Mathematics},
  volume  = {10},
  number  = {4},
  pages   = {1257--1261},
  year    = {1960}
}

@article{vaidya1996new,
  author  = {Vaidya, Pravin M.},
  title   = {A new algorithm for minimizing convex functions over convex sets},
  journal = {Mathematical Programming},
  volume  = {73},
  number  = {3},
  pages   = {291--341},
  year    = {1996}
}

@article{bertsimas2004solving,
  author  = {Bertsimas, Dimitris and Vempala, Santosh},
  title   = {Solving convex programs by random walks},
  journal = {Journal of the ACM},
  volume  = {51},
  number  = {4},
  pages   = {540--556},
  year    = {2004}
}

@inproceedings{lee2015faster,
  author    = {Lee, Yin Tat and Sidford, Aaron and Wong, Sam Chiu-wai},
  title     = {A faster cutting plane method and its implications for combinatorial and convex optimization},
  booktitle = {IEEE 56th Annual Symposium on Foundations of Computer Science (FOCS)},
  pages     = {1049--1065},
  year      = {2015}
}

@inproceedings{jiang2020improved,
  author    = {Jiang, Haotian and Lee, Yin Tat and Song, Zhao and Wong, Sam Chiu-wai},
  title     = {An improved cutting plane method for convex optimization, convex-concave games, and its applications},
  booktitle = {Proceedings of the 52nd Annual ACM SIGACT Symposium on Theory of Computing (STOC)},
  pages     = {944--953},
  year      = {2020}
}

@inproceedings{lee2018efficient,
  author    = {Lee, Yin Tat and Sidford, Aaron and Vempala, Santosh S.},
  title     = {Efficient convex optimization with membership oracles},
  booktitle = {Proceedings of the 31st Conference on Learning Theory (COLT)},
  series    = {Proceedings of Machine Learning Research},
  volume    = {75},
  pages     = {1292--1294},
  year      = {2018}
}

@article{rodomanov2023subgradient,
  author  = {Rodomanov, Anton and Nesterov, Yurii},
  title   = {Subgradient ellipsoid method for nonsmooth convex problems},
  journal = {Mathematical Programming},
  volume  = {199},
  pages   = {305--341},
  year    = {2023}
}

@article{nemirovski2010accuracy,
  author  = {Nemirovski, Arkadi and Onn, Shmuel and Rothblum, Uriel G.},
  title   = {Accuracy certificates for computational problems with convex structure},
  journal = {Mathematics of Operations Research},
  volume  = {35},
  number  = {1},
  pages   = {52--78},
  year    = {2010}
}

@book{nesterov2018lectures,
  author    = {Nesterov, Yurii},
  title     = {Lectures on Convex Optimization},
  edition   = {2},
  series    = {Springer Optimization and Its Applications},
  volume    = {137},
  publisher = {Springer},
  year      = {2018}
}

@article{nesterov1983method,
  author  = {Nesterov, Yu. E.},
  title   = {A method for solving the convex programming problem with convergence rate {$O(1/k^2)$}},
  journal = {Soviet Mathematics Doklady},
  volume  = {27},
  pages   = {372--376},
  year    = {1983}
}

@article{sion1958general,
  author  = {Sion, Maurice},
  title   = {On general minimax theorems},
  journal = {Pacific Journal of Mathematics},
  volume  = {8},
  number  = {1},
  pages   = {171--176},
  year    = {1958}
}

@article{devolder2014first,
  author  = {Devolder, Olivier and Glineur, Fran{\c{c}}ois and Nesterov, Yurii},
  title   = {First-order methods of smooth convex optimization with inexact oracle},
  journal = {Mathematical Programming},
  volume  = {146},
  number  = {1--2},
  pages   = {37--75},
  year    = {2014}
}

@article{stonyakin2021inexact,
  author  = {Stonyakin, Fedor and Tyurin, Alexander and Gasnikov, Alexander and Dvurechensky, Pavel and Agafonov, Artem and Dvinskikh, Darina and Alkousa, Mohammad and Pasechnyuk, Dmitry and Artamonov, Sergei and Piskunova, Victoria},
  title   = {Inexact model: A framework for optimization and variational inequalities},
  journal = {Optimization Methods and Software},
  volume  = {36},
  number  = {6},
  pages   = {1155--1201},
  year    = {2021}
}

@article{gladin2023solving,
  author  = {Alkousa, Mohammad and Gasnikov, Alexander and Gladin, Egor and Kuruzov, Ilya and Pasechnyuk, Dmitry and Stonyakin, Fedor},
  title   = {Solving strongly convex-concave composite saddle-point problems with low dimension of one group of variable},
  journal = {Sbornik: Mathematics},
  volume  = {214},
  number  = {3},
  pages   = {285--333},
  year    = {2023}
}

@article{gladin2021solving,
  author  = {Gladin, E. and Alkousa, M. and Gasnikov, A.},
  title   = {Solving convex min-min problems with smoothness and strong convexity in one group of variables and low dimension in the other},
  journal = {Automation and Remote Control},
  volume  = {82},
  number  = {10},
  pages   = {1679--1691},
  year    = {2021}
}

@article{gladin2022vaidya,
  author  = {Gladin, E. L. and Gasnikov, A. V. and Ermakova, E. S.},
  title   = {Vaidya's method for convex stochastic optimization problems in small dimension},
  journal = {Mathematical Notes},
  volume  = {112},
  number  = {1--2},
  pages   = {183--190},
  year    = {2022}
}

@inproceedings{gladin2023algorithm,
  author    = {Gladin, Egor and Lavrik-Karmazin, Maksim and Zainullina, Karina and Rudenko, Varvara and Gasnikov, Alexander and Tak{\'a}{\v{c}}, Martin},
  title     = {Algorithm for constrained {M}arkov decision process with linear convergence},
  booktitle = {Proceedings of the 26th International Conference on Artificial Intelligence and Statistics (AISTATS)},
  series    = {Proceedings of Machine Learning Research},
  volume    = {206},
  pages     = {11506--11533},
  year      = {2023}
}

@article{pasechnyuk2019one,
  author  = {Pasechnyuk, D. A. and Stonyakin, F. S.},
  title   = {One method for minimization a convex {L}ipschitz-continuous function of two variables on a fixed square},
  journal = {Computer Research and Modeling},
  volume  = {11},
  number  = {3},
  pages   = {379--395},
  year    = {2019}
}

@article{alkousa2020accelerated,
  author  = {Alkousa, M. S. and Gasnikov, A. V. and Dvinskikh, D. M. and Kovalev, D. A. and Stonyakin, F. S.},
  title   = {Accelerated methods for saddle-point problem},
  journal = {Computational Mathematics and Mathematical Physics},
  volume  = {60},
  number  = {11},
  pages   = {1787--1809},
  year    = {2020}
}

@article{borodich2023optimal,
  author  = {Borodich, Ekaterina and Kormakov, Georgiy and Kovalev, Dmitry and Beznosikov, Aleksandr and Gasnikov, Alexander},
  title   = {Near-optimal algorithm with complexity separation for strongly convex-strongly concave composite saddle point problems},
  journal = {Optimization Methods and Software},
  doi     = {10.1080/10556788.2025.2545846},
  year    = {2025}
}

@inproceedings{kovalev2024linear,
  author    = {Borodich, Ekaterina and Gasnikov, Alexander and Kovalev, Dmitry},
  title     = {On linear convergence in smooth convex-concave bilinearly-coupled saddle-point optimization: Lower bounds and optimal algorithms},
  booktitle = {Proceedings of the 42nd International Conference on Machine Learning (ICML)},
  series    = {Proceedings of Machine Learning Research},
  volume    = {267},
  pages     = {5045--5100},
  year      = {2025}
}

@inproceedings{kovalev2022accelerated,
  author    = {Kovalev, Dmitry and Gasnikov, Alexander and Richt{\'a}rik, Peter},
  title     = {Accelerated primal-dual gradient method for smooth and convex-concave saddle-point problems with bilinear coupling},
  booktitle = {Advances in Neural Information Processing Systems (NeurIPS)},
  volume    = {35},
  year      = {2022}
}

@inproceedings{kovalev2022first,
  author    = {Kovalev, Dmitry and Gasnikov, Alexander},
  title     = {The first optimal acceleration of high-order methods in smooth convex optimization},
  booktitle = {Advances in Neural Information Processing Systems (NeurIPS)},
  volume    = {35},
  year      = {2022}
}

@article{nesterov2021implementable,
  author  = {Nesterov, Yurii},
  title   = {Implementable tensor methods in unconstrained convex optimization},
  journal = {Mathematical Programming},
  volume  = {186},
  pages   = {157--183},
  year    = {2021}
}

@inproceedings{carmon2022optimal,
  author    = {Carmon, Yair and Hausler, Danielle and Jambulapati, Arun and Jin, Yujia and Sidford, Aaron},
  title     = {Optimal and adaptive {M}onteiro-{S}vaiter acceleration},
  booktitle = {Advances in Neural Information Processing Systems (NeurIPS)},
  volume    = {35},
  year      = {2022}
}

@article{korpelevich1976extragradient,
  author  = {Korpelevich, G. M.},
  title   = {The extragradient method for finding saddle points and other problems},
  journal = {Matecon},
  volume  = {12},
  pages   = {747--756},
  year    = {1976}
}

@article{nemirovski2004prox,
  author  = {Nemirovski, Arkadi},
  title   = {Prox-method with rate of convergence {$O(1/t)$} for variational inequalities with {L}ipschitz continuous monotone operators and smooth convex-concave saddle point problems},
  journal = {SIAM Journal on Optimization},
  volume  = {15},
  number  = {1},
  pages   = {229--251},
  year    = {2004}
}

@article{nesterov2005smooth,
  author  = {Nesterov, Yu.},
  title   = {Smooth minimization of non-smooth functions},
  journal = {Mathematical Programming},
  volume  = {103},
  number  = {1},
  pages   = {127--152},
  year    = {2005}
}

@article{zhang2022lower,
  author  = {Zhang, Junyu and Hong, Mingyi and Zhang, Shuzhong},
  title   = {On lower iteration complexity bounds for the convex concave saddle point problems},
  journal = {Mathematical Programming},
  volume  = {194},
  pages   = {901--935},
  year    = {2022}
}

@article{lan2016gradient,
  author  = {Lan, Guanghui},
  title   = {Gradient sliding for composite optimization},
  journal = {Mathematical Programming},
  volume  = {159},
  pages   = {201--235},
  year    = {2016}
}

@inproceedings{thekumparampil2022lifted,
  author    = {Thekumparampil, Kiran Koshy and He, Niao and Oh, Sewoong},
  title     = {Lifted primal-dual method for bilinearly coupled smooth minimax optimization},
  booktitle = {Proceedings of the 25th International Conference on Artificial Intelligence and Statistics (AISTATS)},
  series    = {Proceedings of Machine Learning Research},
  volume    = {151},
  pages     = {4281--4308},
  year      = {2022}
}

@inproceedings{jin2022sharper,
  author    = {Jin, Yujia and Sidford, Aaron and Tian, Kevin},
  title     = {Sharper rates for separable minimax and finite sum optimization via primal-dual extragradient methods},
  booktitle = {Proceedings of the 35th Conference on Learning Theory (COLT)},
  series    = {Proceedings of Machine Learning Research},
  volume    = {178},
  pages     = {4362--4415},
  year      = {2022}
}

@inproceedings{carmon2021thinking,
  author    = {Carmon, Yair and Jambulapati, Arun and Jin, Yujia and Sidford, Aaron},
  title     = {Thinking inside the ball: Near-optimal minimization of the maximal loss},
  booktitle = {Proceedings of the 34th Conference on Learning Theory (COLT)},
  series    = {Proceedings of Machine Learning Research},
  volume    = {134},
  pages     = {866--882},
  year      = {2021}
}

@inproceedings{carmon2020acceleration,
  author    = {Carmon, Yair and Jambulapati, Arun and Jiang, Qijia and Jin, Yujia and Lee, Yin Tat and Sidford, Aaron and Tian, Kevin},
  title     = {Acceleration with a ball optimization oracle},
  booktitle = {Advances in Neural Information Processing Systems (NeurIPS)},
  volume    = {33},
  pages     = {19052--19063},
  year      = {2020}
}

@inproceedings{sagawa2020distributionally,
  author    = {Sagawa, Shiori and Koh, Pang Wei and Hashimoto, Tatsunori B. and Liang, Percy},
  title     = {Distributionally robust neural networks for group shifts: On the importance of regularization for worst-case generalization},
  booktitle = {International Conference on Learning Representations (ICLR)},
  year      = {2020}
}

@inproceedings{agarwal2018reductions,
  author    = {Agarwal, Alekh and Beygelzimer, Alina and Dud{\'\i}k, Miroslav and Langford, John and Wallach, Hanna},
  title     = {A reductions approach to fair classification},
  booktitle = {Proceedings of the 35th International Conference on Machine Learning (ICML)},
  series    = {Proceedings of Machine Learning Research},
  volume    = {80},
  pages     = {60--69},
  year      = {2018}
}

@inproceedings{cotter2019two,
  author    = {Cotter, Andrew and Jiang, Heinrich and Sridharan, Karthik},
  title     = {Two-player games for efficient non-convex constrained optimization},
  booktitle = {Proceedings of the 30th International Conference on Algorithmic Learning Theory (ALT)},
  series    = {Proceedings of Machine Learning Research},
  volume    = {98},
  pages     = {300--332},
  year      = {2019}
}

@inproceedings{kohavi1996scaling,
  title     = {Scaling up the accuracy of {Naive-Bayes} classifiers: a decision-tree hybrid},
  author    = {Kohavi, Ron},
  booktitle = {Proceedings of the Second International Conference on Knowledge Discovery and Data Mining},
  pages     = {202--207},
  year      = {1996}
}

@article{lemarechal1995new,
  author  = {Lemar{\'e}chal, Claude and Nemirovskii, Arkadii and Nesterov, Yurii},
  title   = {New variants of bundle methods},
  journal = {Mathematical Programming},
  volume  = {69},
  number  = {1--3},
  pages   = {111--147},
  year    = {1995}
}

@article{kiwiel2006proximal,
  author  = {Kiwiel, Krzysztof C.},
  title   = {A proximal bundle method with approximate subgradient linearizations},
  journal = {SIAM Journal on Optimization},
  volume  = {16},
  number  = {4},
  pages   = {1007--1023},
  year    = {2006}
}

@article{deoliveira2014convex,
  author  = {de Oliveira, Welington and Sagastiz{\'a}bal, Claudia and Lemar{\'e}chal, Claude},
  title   = {Convex proximal bundle methods in depth: a unified analysis for inexact oracles},
  journal = {Mathematical Programming},
  volume  = {148},
  number  = {1--2},
  pages   = {241--277},
  year    = {2014}
}

@article{deoliveira2014level,
  author  = {de Oliveira, Welington and Sagastiz{\'a}bal, Claudia},
  title   = {Level bundle methods for oracles with on-demand accuracy},
  journal = {Optimization Methods and Software},
  volume  = {29},
  number  = {6},
  pages   = {1180--1209},
  year    = {2014}
}

@article{vanackooij2014level,
  author  = {van Ackooij, Wim and de Oliveira, Welington},
  title   = {Level bundle methods for constrained convex optimization with various oracles},
  journal = {Computational Optimization and Applications},
  volume  = {57},
  number  = {3},
  pages   = {555--597},
  year    = {2014}
}

@article{necoara2014rate,
  author  = {Necoara, Ion and Nedelcu, Valentin},
  title   = {Rate analysis of inexact dual first-order methods application to dual decomposition},
  journal = {IEEE Transactions on Automatic Control},
  volume  = {59},
  number  = {5},
  pages   = {1232--1243},
  year    = {2014}
}

@inproceedings{ji2021bilevel,
  author    = {Ji, Kaiyi and Yang, Junjie and Liang, Yingbin},
  title     = {Bilevel optimization: Convergence analysis and enhanced design},
  booktitle = {Proceedings of the 38th International Conference on Machine Learning (ICML)},
  series    = {Proceedings of Machine Learning Research},
  volume    = {139},
  pages     = {4882--4892},
  year      = {2021}
}

@inproceedings{lin2020near,
  author    = {Lin, Tianyi and Jin, Chi and Jordan, Michael I.},
  title     = {Near-optimal algorithms for minimax optimization},
  booktitle = {Proceedings of the 33rd Conference on Learning Theory (COLT)},
  series    = {Proceedings of Machine Learning Research},
  volume    = {125},
  pages     = {2738--2779},
  year      = {2020}
}

@inproceedings{kovalev2022optimal,
  author    = {Kovalev, Dmitry and Gasnikov, Alexander},
  title     = {The first optimal algorithm for smooth and strongly-convex-strongly-concave minimax optimization},
  booktitle = {Advances in Neural Information Processing Systems (NeurIPS)},
  volume    = {35},
  pages     = {14691--14703},
  year      = {2022}
}

@article{gasnikov2018universal,
  author  = {Gasnikov, A. V. and Nesterov, Yu. E.},
  title   = {Universal method for stochastic composite optimization problems},
  journal = {Computational Mathematics and Mathematical Physics},
  volume  = {58},
  number  = {1},
  pages   = {48--64},
  year    = {2018}
}

@inproceedings{gladin2021mixed,
  author    = {Gladin, Egor and Sadiev, Abdurakhmon and Gasnikov, Alexander and Dvurechensky, Pavel and Beznosikov, Aleksandr and Alkousa, Mohammad},
  title     = {Solving smooth min-min and min-max problems by mixed oracle algorithms},
  booktitle = {Mathematical Optimization Theory and Operations Research: Recent Trends (MOTOR 2021)},
  series    = {Communications in Computer and Information Science},
  pages     = {19--40},
  publisher = {Springer},
  year      = {2021},
  doi       = {10.1007/978-3-030-86433-0_2}
}

@article{gladin2023accuracy,
  author  = {Gladin, Egor and Gasnikov, Alexander and Dvurechensky, Pavel},
  title   = {Accuracy certificates for convex minimization with inexact oracle},
  journal = {Journal of Optimization Theory and Applications},
  volume  = {204},
  number  = {1},
  year    = {2025},
  doi     = {10.1007/s10957-024-02599-9}
}

@inproceedings{rademacher2007approximating,
  author    = {Rademacher, Luis A.},
  title     = {Approximating the centroid is hard},
  booktitle = {Proceedings of the 23rd Annual Symposium on Computational Geometry (SoCG)},
  pages     = {302--305},
  year      = {2007}
}

\end{document}